\documentclass[a4paper,11pt,twoside]{article}

\usepackage{amsmath, amssymb, amsthm}
\usepackage{color}

\usepackage{graphicx}

\theoremstyle{definition}

\newtheorem{thm}{Theorem}[section]

\newtheorem{lemma}{Lemma}[section]
\newtheorem{cor}{Corollary}[section]
\newtheorem{remark}{Remark}[section]
\newtheorem{example}{Example}[section]

\newtheorem{as}{A\!}
\newtheorem{asb}{B\!}

\numberwithin{equation}{section}
\allowdisplaybreaks

\def\dis{\displaystyle}
\def\R{\mathbb{R}}

\def\N{\mathbb{N}}

\def\e{{\epsilon}}
\def\D{\Delta}
\def\d{\delta}
\def\l{\left}
\def\r{\right}
\def\a{\alpha}
\def\b{\beta}
\def\G{\Gamma}
\def\g{\gamma}
\def\th{\vartheta}
\def\s{\sigma}
\def\z{\zeta}

\def\la{\lambda}

\def\vp{\varphi}
\def\n{\nabla}
\def\p{\partial}
\def\F{\mathcal{F}}

\def\P{{\mathbb{P}}}
\def\toP{\stackrel{p}{\longrightarrow}}
\def\toD{\stackrel{d}{\longrightarrow}}
\def\E{\mathbb{E}}

\def\mb#1{\mbox{\boldmath $#1$}}

\def\I{\mathbf{1}}

\def\wh#1{\widehat{#1}} 
 
\def\wt#1{\widetilde{#1}}

\def\df{\mathrm{d}}

\title{Asymptotic distributions for estimated expected functionals of general random elements}
\author{Yasutaka Shimizu\footnote{E-mail: {\tt shimizu@waseda.jp}} \\ 
{\it Department of Applied Mathematics, Waseda University; }\\
JST CREST.}
\date{August 8, 2020}

\begin{document}

\maketitle 

\begin{abstract} 

We consider an estimation problem of expected functionals of a general random element that values in a metric space. If the functional is written by an explicit function of some unknown parameters, we can estimate it by plugging-in a suitable estimator into the function, and we can find the asymptotic distribution by a well-known delta method. However, if the functional is implicit in the parameters, it causes a problem of specifying asymptotic distribution.  
This paper gives a general condition to specify the asymptotic distribution even if the functional is implicit in the parameters, and further investigates it in detail when the random elements are semimartingales with jumps. 

\begin{flushleft}
\vspace{1mm} 
\end{flushleft}
\end{abstract}

\section{Introduction}

Let $(\Omega,\F,\P)$ be a probability space, and ${\cal X}$ be a metric space with norm $\|\cdot \|$.  
Consider a ${\cal X}$-valued random element $X^{\th}$ with an unknown parameter $\th\in \Theta\subset \R^p$, and the distribution of $X^\th$ is $P_{\th}:=\P\circ (X^\th)^{-1}$. 
Suppose that there exists the true value $\th_0\in \Theta$, and we are interested in inference for the following expected functional of $X^\th$:  
\begin{align*}
H(\th_0) = \E\l[h(X^{\th_0},\th_0)\r] = \int_{{\cal X}} h(x,\th_0)\,P_{\th_0}(\df x), 
\end{align*}
where $h: {\cal X}\times \Theta \to \R$. 

Such a expected functional appears in many statistical problem, 
where ${\cal X}$ is not only an Euclidean space, but also a functional space. 
When $H(\th)$ is written in explicit function of $\th$, we can estimate it as a {\it plug-in estimator} $H(\wh{\th})$ with a suitable estimator of $\th_0$ plugged-in, and it will be easy to evaluate the statistical error by, e.g., the {\it delta method}: for example, if $H$ is differentiable and $\wh{\th}$ is asymptotically normal with asymptotic variance $\s^2_0$, then $H(\wh{\th})$ is also asymptotically normal with asymptotic  variance $|\n_\th H(\th_0)|^2\s^2_0$; see, e.g., Corollary \ref{cor:mc-est}, below. 
However, it is not practicable when $H(\th)$ is implicit in $\th$, which is our interest in this paper. 

For example, consider a case where $X^\th=(X^\th_t)_{t\in [0,T]}$ is a diffusion process, which is a $C([0,T])$-valued random element. Estimating problem for 
\begin{align}
H(\th) = \E\l[e^{-r T}\max (X^{\th}_T  - K,0)\r],\quad r,\, K>0  \label{call-option}
\end{align}
will appear in a typical financial problem for the pricing an {\it European call option} with strike price $K$ and interest rate $r$ when $X^\th$ is a stock price. 
This functional $H$ is generally implicit in $\th$ except for some special models of $X^\th$. In such a case, Monte Carlo simulation will be used: based on an estimated value of $\th_0$, say $\wh{\th}$, generate many paths from the estimated distribution $P_{\wh{\th}}$ by simulations, say $X^{\wh{\th}}(i)\ (i=1,2,\dots, B)$, and compute the average $B^{-1}\sum_{i=1}^B e^{-r T}\max (X^{\wh{\th}}_T(i)  - K,0)$, which goes to, not $H(\th_0)$, but $H(\wh{\th})$ as $B\to \infty$, that remains a statistical error. To get information of $H(\th_0)$, we need to know the asymptotic distribution of $H(\wh{\th})$, but such a statistical error seems often ignored in practice since the asymptotic variance of $H(\wh{\th})$ is not clear. 

In this paper, we investigate the specification of asymptotic distribution of $H(\wh{\th})$ even in the case where $H(\th)$ is possibly implicit in $\th$ (so the delta method is not available explicitly), when $X^\th$ is a random element that  values in a general metric space $({\cal X},\|\cdot\|)$. More precisely, we will find the asymptotic distribution of 
\begin{align}
\g_n^{-1}(H(\wh{\th}) - H(\th_0)) \label{eq:target}
\end{align}
when $\g_n^{-1}(\wh{\th} - \th_0) \toD Z$ as $n\to \infty$ for some random variable $Z$ and norming sequence $\g_n\to 0\ (n\to \infty)$. 
This problem is a very fundamental problem in statistics, and 
it is well known, in the case where $H$ is explicit and differentiable, 
that the asymptotic distribution is found by the {\it delta method}: 
\[
\g_n^{-1}(H(\wh{\th}) - H(\th_0)) \toD \n_\th H(\th_0) Z. 
\]
However, it seems not discussed sufficiently in a statistical context when $H(\th)$ is implicit, and especially when $X^\th$ is a stochastic process. 
Formally speaking, we need the following derivative: 
\[
\n_\th H(\th) = \n_\th \E[h(X^\th,\th)] = \E\l[\n_xh(X^\th,\th)\n_\th X^\th + \dot{h}(X^\th,\th)\r], 
\]
where $\dot{h}(x,\th) = \n_\th h(x,\th)$, if $\n_\th$ and $\E$ are exchangable. 
However, the last expectation has a bit problem sisnce the sense of the derivative ``$\n_\th X^\th$" is still not clear. 

Such a ``derivative" has been considered in sensitivity analysis for expected functionals. 
For example, when $X^\th$ is a Euclidian valued random variable with a parameter $\th$ in the distribution, Suri \cite{s83} discusses an expression of a derivative $\n_\th X^\th$ based on the distribution function. 
Moreover, when $X^\th$ is a stochastic process with the initial value $X^\th_0=\th$, 
the map $x\mapsto X^x_t(\omega)$ is called a {\it stochastic flow}, and the continuity and the differentiability with respect to $\th$ can be discussed; see., e.g., Protter \cite{p05}, Chapter V.7 in the case where $X^x$ satisfies a stochastic differential equation. 

A similar problem appears in financial computation of {\it Greeks}, which are obtained as derivatives of option prices with respect to some specific parameters since those prices are written by expected functionals of underlying stock prices as in, e.g., \eqref{call-option}. 
This problem is recently well studied via {\it Malliavin calculus}, which has a powerful tool such as ``integration-by-parts" formula to compute such derivatives; see, e.g., Davis and Johansson \cite{dj06}, Fourni\'e {\it et al.} \cite{fetal99, fetal01},  Gobet and Kohatsu-Higa \cite{gk03}, among others. see also Kohatsu-Higa and Montero \cite{km04} as a good guidance.  t
On the other hand, Chen and Glasserman \cite{cg07},  Glasserman and Liu \cite{gl10} take the {\it path-wise derivative approach} to compute $\E\l[\n_xh(X^\th,\th)\n_\th X^\th\r]$ by Monte Carlo simulations; see also Glasserman \cite{gl04}, Chapter 7. 
 
In this paper, we will take a different approach. We consider a differentiability of $X^\th$ with respect to $\th$ in the $L^q$-sense for $q>0$ to evaluate errors in higher order terms; see the condition A4($q$), below. Under this approach,  the regularity conditions for \eqref{eq:target} can be an expectation-based and easy to check, 

First, we will discuss general conditions to yield asymptotic distributions in both cases where ${\cal X}=\R^d$ and ${\cal X}$ is a functional space, and then we see the each case in detail. 
The former is a standard situation, where the delta method is obtained as a special case; see Corollary \ref{cor:mc-est},  but the latter is performed with a kind of {\it derivative} of $X^\th$ with respect to $\th$; see Theorem \ref{thm:cond}. The case where $X^\th$ are semimartingale that values in $\mathbb{D}$-space is important in applications. In this case, the asymptotic distribution can be described in terms of the {\it derivative process} of $X^\th$ with respect to $\th$ in $L^q$-sense. 
The derivative process is essentially the same as the path-wise derivative discussed in  Chen and Glasserman \cite{cg07}, but we will give a different approach and an evaluation for not only a continuous diffusion processes, but also semimartingales with jumps, which is a new contribution on the derivative process because Glasserman and Liu \cite{gl10} just discuss from the simulations point of view.  
Our investigation on jump processes indicates that an “error” in \eqref{eq:target} may get worse when $X^\th$ values in $\mathbb{D}$-space than the case where $X^\th$ values in $\mathbb{C}$-space; see Remark \ref{rem:unknown2}.

The paper is organized as follows. In Section \ref{sec:mc}, we shall state fundamental conditions to get the asymptotic distribution in a general formulation, and a special case where ${\cal X} = \R^d$ is described there as a corollary of the general statement as well as the usual delta method. 
In later sections, we will consider more specific cases. In Section \ref{sec:C-space}, we consider the case where $X^\th$ is functional valued, and a sufficient condition to ensure that the asymptotic normality of $H(\wh{\th})$ is given in terms of the norm of the functional space ${\cal X}$. 
We shall check the condition in each specific form of the functional. 
Section \ref{sec:sde} is devoted to the case where $X^\th$ is described by stochastic differential equations with jumps. The situation differs to a large extent when $X^\th$ does not have a jump in the path (${\cal X}$ is a $\mathbb{C}$-space), compared to when $X^\th$ does (${\cal X}$ is a $\mathbb{D}$-space).  
The result indicates that we should be careful to use $H(\wh{\th})$ since it may not be asymptotically normal in the case where $X^\th$ is a jump process.

Throughout the paper, we use the following notation. 
\begin{itemize} 
\item $A\lesssim B$ means that there exists a universal constant $c>0$ such that $A \le c\cdot B$. 
\item A $d$-dim Gaussian variable (distribution) with mean $0$ and variance-covariance matrix $\Sigma$ is denoted by $N_d(0,\Sigma)$. We omit the index $d=1$. 
\item For a function $f: \R^d\to \R$ and $x=(x_1,\dots,x_d)\in \R^d$, 
\[
\n_x f = \l(\frac{\p f}{\p x_1},\dots, \frac{\p f}{\p x_d} \r)^\top, 
\]
and $\n_x^k = \n_x\otimes \n_x^{k-1},\ (k=2,3,\dots)$, constitutes a multilinear form. 

\item For a function $f:\R^d\times \Theta \to \R$ and an integer $k$, 
\begin{align*}
\dot{f}(x,\th) = \n_\th f(x,\th); \quad f^{(k)}(x,\th) = \n_x^k f(x,\th).  
\end{align*}
Note that $\n_x^k f$ is a $k$-th order tensor. 

\item For a $k$-th order tensor $x=(x_{i_1,i_2,\dots,i_k})_{i_1,\dots,i_k = 1,\dot d} \in \R^d\otimes \dots \otimes \R^d$, 
\[
|x| = \sqrt{\sum_{i_1=1}^d \dots\sum_{i_k=1}^d x^2_{i_1,i_2,\dots,i_k}}. 
\] 
\item For a ${\cal X}$-valued random element $X$, $\|X\|_{L^p}=\l(\E\|X\|^p\r)^{1/p}$ for $p>0$, where 
$\|\cdot \|$ a norm on ${\cal X}$, and write $X\in L^p$ if $\|X\|_{L^p} < \infty$. 
\end{itemize}

\section{Fundamental conditions for asymptotic distributions}\label{sec:mc}

\subsection{Basic results in general formulation}
Assume that a realization of $X^{\th_0}$ from $P_{\th_0}$, say $X^{\th_0,n}$,  is given, where $n$ is supposed to be a parameter on which the sample size depends. For example, when we observe $n$-samples of i.i.d. variables $\{X_k\}_{k\in \N}$, it can be regarded as $X^{\th_0, n} =(X_1,X_2,\dots,X_n)$, so $n$ represents the number of samples. When $X^{\th_0}$ is a stochastic process $X=(X_t)_{t\ge 0}$, $X^{\th_0,n}$ can be a time-continuous observation in a $[0,n]$-time interval: $X^{\th_0,n}=(X_t)_{t\in [0,n]}$, or it can be discrete samples such as $X^{\th_0,n}=(X_0,X_{t_1},\dots,X_{t_n})$, among others. 
We assume that a ``good" estimator of $\th_0$ is given based on the observations $X^{\th_0,n}$, say 
\[
\wh{\th}_n:=\wh{\th}(X^{\th_0,n}). 
\]
We assume that some estimator of $\th_0$, say $\wh{\th}_n$, is given in a suitable manner. 
We shall investigate a fundamental condition under which $H(\wh{\th}_n)$ is the asymptotic distribution is specified. 

We make the following conditions. 
\begin{as} \label{as:h-deriv}
For any $\th'\in \Theta$,  $\dis \n_\th \E[h(X^{\th'},\th)] = \E[\dot{h}(X^{\th'},\th)]$. 
\end{as}

\begin{as}\label{as:h-cont}
The function $\th \mapsto \E[\dot{h}(X^\th,\th)]$ is continuous. 
\end{as}

\begin{as}\label{as:Z}
There exists a diagonal matrix  $\G_n=\mathrm{diag}(\g_n^{(1)},\dots,\g_n^{(p)})$ with $\g_n^{(k)} > 0$ and 
$\g_{n*}:=\max_{1\le k\le p} \g_n^{(k)} \downarrow 0\ (n\to \infty)$ such that the estimator $\wh{\th}_n$ satisfies 
\[
\G_n^{-1}(\wh{\th}_n - \th_0) \toD Z;\qquad \g_{n*}^{-1}(\wh{\th}_n - \th_0) \toD Z^*, 
\]
as $n\to \infty$, for $p$-dim random variables $Z$ and $Z^*$. 
\end{as}

\begin{flushleft}
{\bf A4($\mb{q}$).}\ \ There exists  a ${\cal X}^p$-valued random element $Y^\th$ such that 
$Y^\th \in L^q$ for $q>0$ and 
\begin{align*}
\|X^{\th + u} - X^\th - u^\top Y^\th\|_{L^q} =o(|u|),\quad |u|\to 0,  
\end{align*}
uniformly in $\th\in \Theta$. 
\end{flushleft}

\begin{remark}
As for condition A\ref{as:Z}, it usually holds that $\g_n^{(k)}=1/\sqrt{n}$ for all $k$ in i.i.d.-cases, but there are some examples where the rates of convergence are different among parameters, e.g., for a sequence such that $T_n/n\to 0$ as $n\to \infty$ and constants $\s_1^2,\s_2^2\ne 0$, 
\[
\mathrm{diag}(\sqrt{T_n},\sqrt{n})(\wh{\th}_n^{(1)} - \th_0^{(1)}, \wh{\th}_n^{(2)} - \th_0^{(2)})^\top \toD N_2\l(0,\mathrm{diag}(\s_1^2,\s_2^2)\r)=Z. 
\]
In such a case, $\g_{n*}=1/\sqrt{T_n}$ and we have 
\[
\sqrt{T_n}(\wh{\th}_n^{(1)} - \th_0^{(1)}, \wh{\th}_n^{(2)} - \th_0^{(2)})^\top \toD (N(0,\s_1^2),0)^\top = Z^*, 
\]
which is a degenerate random variable; see also Examples \ref{ex:levy} and \ref{ex:ou2}. 
\end{remark}

\begin{remark}
In condition A4($q$), the random element $Y^\th$ is interpreted as the first derivative of $X^\th$ with respect to $\th$ in the sense of $L^q$. 
\end{remark}

Although the following seems to be a simple result, we shall claim it here since it is a basis of the discussion below. 
\begin{thm}\label{thm:mc-est}
Suppose that A\ref{as:h-deriv} -- A\ref{as:Z} hold true, and that there exists a constant vector $C_{\th_0} \in \R^p$ such that 
for $\g_{n*}:=\max_{1\le k\le p}\g_n^{(k)}$, 
\begin{align}
\g_{n*}^{-1} \E\l[h(X^{\th},\th_0) - h(X^{\th_0},\th_0)\r]\Big|_{\th = \wh{\th}_n} = C_{\th_0}^\top \g_{n*}^{-1}(\wh{\th}_n - \th_0) + o_p(1),\label{term2}
\end{align}
as $n\to \infty$. Then, it holds that 
\begin{align*}
\g_{n*}^{-1} [H(\wh{\th}_n) - H(\th_0)] \toD (\E[\dot{h}(X^{\th_0},\th_0)] + C_{\th_0})^\top \,Z^*, \quad n\to \infty.  
\end{align*}
\end{thm}

\begin{proof}
Let $X_*^{\th_0}\sim P_{\th_0}$, which is independent of the data $X^{\th_0,n}$. Then we have that  
\begin{align*}
H(&\wh{\th}_n) - H(\th_0) \notag \\
&= \E\l[h(X_*^{\wh{\th}_n},\wh{\th}_n) - h(X_*^{\th_0},\th_0) \big| X^{\th_0,n}\r]  \notag \\
&= \E\l[h(X_*^{\wh{\th}_n},\wh{\th}_n) - h(X_*^{\wh{\th}_n},\th_0) \big| X^{\th_0,n}\r] 
+ \E\l[h(X_*^{\wh{\th}_n},\th_0) - h(X_*^{\th_0},\th_0) \big| X^{\th_0,n}\r] \notag \\
&= \E\l[\dot{h}(X_*^{\th},\th_0 + \eta_n(\wh{\th} - \th_0)\r]\Big|_{\th = \wh{\th}_n} (\wh{\th}_n - \th_0) 
+ \E\l[h(X^{\th},\th_0) - h(X^{\th_0},\th_0)\r]\Big|_{\th = \wh{\th}_n},    
\end{align*}
where $\eta_n$ is a random variable values in $[0,1]$. 
We use the mean value theorem in the last equality.  
Then, under A\ref{as:h-cont}, the continuous mapping theorem yields the result. 
\end{proof}

This theorem immediately leads us a version of the {\it delta method} when ${\cal X} = \R^d$. 
\begin{cor}\label{cor:mc-est}
Consider the case where ${\cal X} = \R^d$ and $X^\th$ be a random variable with probability density $f:{\cal X}\times \Theta \to \R$ in Theorem \ref{thm:mc-est}. Suppose that $f$ is twice differentiable with respect to $\th\in \Theta$ with $\int_{\cal X} h(x,\th_0)\dot{f}(x,\th_0)\,\df x < \infty$. Moreover, suppose A\ref{as:Z} holds, and that it holds for the second derivative of $f$ in $\th$, say $\ddot{f}$, such that 
\begin{align}
\sup_{\th\in \Theta} \l|\int_{\cal X} h(x,\th_0)\ddot{f}(x,\th)\,\df x\r| < \infty. \label{f-dd}
\end{align}
Then $C_\th$ in \eqref{term2} is given by $C_\th=\int_\R h(x,\th)\dot{f}(x,\th)\,\df x$, and it follows that 
\[
\g_{n*}^{-1} [H(\wh{\th}_n) - H(\th_0)] \toD  \n_\th H(\th_0)^\top \,Z^*, \quad n\to \infty, 
\]
where 
\[
\n_\th H(\th_0) = \int_{\cal X} \l[\dot{h}(x,\th_0)f(x,\th_0)+ h(x,\th_0)\dot{f}(x,\th_0)\r]\,\df x. 
\]
\end{cor}

\begin{proof}
For $\th \in \Theta$ and $u\in \R^p$ with $\th + u\in \Theta$, it follows from Taylor's formula that 
\begin{align*}
\E\l[h(X^{\th + u},\th) - h(X^\th,\th)\r] &= \int_{\cal X} h(x,\th)\l[f(x,\th + u) - f(x,\th)\r]\,\df x \\
&= \int_{\cal X} h(x,\th)\l[u^\top \dot{f}(x,\th) + u^\top \ddot{f}(x,\th_u) u\r]\,\df x
\end{align*}
where $\th^u:= \th + \eta_u u$ for some $\eta_u\in [0,1]$. 
Hence, when $\th=\th_0$ and $u=\wh{\th}_n - \th_0$ and both sides are multiplied by $\g_{n*}^{-1}$, we have that 
\begin{align*}
\g_{n*}^{-1}\E\l[h(X^{\wh{\th}_n},\th_0) - h(X^{\th_0},\th_0)\r] 
&= \l(\int_{\cal X} h(x,\th_0)\dot{f}(x,\th_0)\,\df x\r)^\top \cdot \g_{n*}^{-1}(\wh{\th}_n - \th_0) \\
&\quad + O_p(|\wh{\th}_n - \th_0|),\quad n\to \infty. 
\end{align*}
Therefore, $C_{\th_0} = \int_{\cal X} h(x,\th_0)\dot{f}(x,\th_0)\,\df x$, and Theorem \ref{thm:mc-est} and the condition A\ref{as:h-deriv} yield that 
\begin{align*}
\g_{n*}^{-1} [H(\wh{\th}_n) - H(\th_0)] &\toD  \l( \int_{\cal X} \l[\dot{h}(x,\th_0)f(x,\th_0)+ h(x,\th_0)\dot{f}(x,\th_0)\r]\,\df x\r)^\top Z^*\\
&= \n_\th H(\th_0)^\top \,Z^*,\quad n\to \infty. 
\end{align*}
\end{proof}

When ${\cal X}$ is not Euclidean, but some functional spaces, the following theorem will be useful to specify the value of $C_{\th_0}$ in Theorem \ref{thm:mc-est}. 

\begin{thm}\label{thm:cond}
Suppose that assumptions A\ref{as:Z} and A4($q$) hold for a constant $q>1$, and that there exists a $\R^p$-valued random variable $G_{\th_0} \in L^1$ such that for each $u\in \R$ with $\th_0 + u\in \Theta$, 
\begin{align}
\l|\E[h(X^{\th_0 + u},\th_0) - h(X^{\th_0},\th_0)  -  u^\top G_{\th_0}]\r| \lesssim \|X^{\th_0 + u} - X^{\th_0} - u^\top Y^{\th_0}\|_{L^q}  +  r_u,  \label{mvt}
\end{align}
where $r_u=o(|u|)$ as $|u|\to 0$. Then the equality \eqref{term2} holds true with $C_{\th_0}=\E[G_{\th_0}]$. 
\end{thm}

\begin{proof}
The assumption A4($q$) with $q>1$ implies that  
\begin{align*}
\l|\E[h(X^{\th_0 + u},\th_0) - h(X^{\th_0},\th_0)  -  u^\top G_{\th_0}]\r| = o(|u|), \quad |u|\to 0. 
\end{align*}
Then, it follows that 
\[
\E[h(X^{\th_0 + u},\th_0) - h(X^{\th_0},\th_0) ] =  \E[G_{\th_0}]^\top u + o(|u|), \quad |u|\to 0. 
\]
When $u = \wh{\th}_n - \th_0$ and both sides are multiplied by $\g_{n*}^{-1}$, 
we obtain
\begin{align*}
\g_{n*}^{-1}\E\l[h(X^{\th},\th_0) - h(X^{\th},\th_0)\r] \Big|_{\th=\wh{\th}_n}  
&= \E[G_\th]^\top  \g_{n*}^{-1}(\wh{\th}_n - \th_0) + o_p(|\g_{n*}^{-1}(\wh{\th}_n - \th_0)|). 
\end{align*}
The last term converges to zero in probability under A\ref{as:Z}. This ends the proof. 
\end{proof}

\begin{example}
Consider a random variable $X^\th$ values on ${\cal X}=\R$ with distribution function $F_\th:\R \to [0,1]$ and a parameter $\th\in \Theta \subset \R$. Suppose that a positive density $\n_x F_\th(x) = f(x,\th)$ exists, and that $\n_\th^2 F_\th^{-1}$ is bounded for simplicity: 
\begin{align}
D:=\sup_{x\in [0,1],\th\in \Theta}|\n_\th^2 F_\th^{-1}(x)| < \infty. \label{bdd}
\end{align}
We shall consider a ``derivative $\n_\th X^\th$" after the idea by Suri \cite{s83}: we may set 
\[
X_\th = F^{-1}_\th(U);\quad \n_\th X^\th :=\n_\th F_\th^{-1}(U)
\]
where $U$ is a uniform random variable on $[0,1]$ independent of $\th$. 
This leads us to $F_\th(F^{-1}_\th(U)) = U$. Differentiating the both sides, we have that $\n_\th F(X^\th) + f(X^\th,\th) \n_\th X^\th = 0$, and that 
\[
\n_\th X^\th = -\frac{\n_\th F_\th(X^\th)}{f(X^\th, \th)}. 
\]
Under the assumption \eqref{bdd}, we can easily see by Taylor's formula that, for any $q>0$, 
\[
\|X^{\th + u} - X^\th - u^\top\n_\th X^\th\|^q \le \frac{D}{2}|u|^2 = o(u),\quad |u|\to 0, 
\]
which yields the condition A($q$) with $Y^\th = \n_\th X^\th$. 

Now we suppose that $\n_x h(X^{\th_0},\th_0) \in L^r$ for some $r>1$ with $1/q + 1/r =1$. Then we can see that 
\[
G_{\th_0} = \n_x h(X^{\th_0},\th_0)\n_\th X^{\th_0}. 
\]
Indeed, we see by Taylor's formula that 
\begin{align*}
\l| \E \l[h(X^{\th_0 + u},\th_0) - h(X^{\th_0},\th_0) - u G_{\th_0} \r] \r|
&= \l| \E\l[\n_x h(X^{\th_0},\th_0) (X^{\th_0+u} - X^{\th_0}) - uG_{\th_0}\r] \r| \\
&\le \E\l[|\n_x h(X^{\th_0},\th_0)| \cdot |X^{\th_0+u} - X^{\th_0} - Y^{\th_0}|\r] \\
&\le \|\n_x h(X^{\th_0},\th_0)\|_{L^r} \|X^{\th_0+u} - X^{\th_0} - Y^{\th_0}\|_{L^q}, 
\end{align*}
under the assumption A\ref{as:h-deriv}, which yields the inequality \eqref{mvt}. 
Therefore it follows by the integration-by-parts that 
\begin{align*}
C_{\th_0} &= \E[G_{\th_0}] = - \int_\R h(x,\th_0) \frac{\n_\th F_{\th_0}(x)}{f(x,\th_0)}\,f(x,\th_0)\,\df x 
 = \int_\R h(x,\th_0)\dot{f}(x,\th_0)\,\df x, 
\end{align*}
which coincides with the expression of $C_{\th_0}$ in Corollary \ref{cor:mc-est}. 
\end{example}


\section{Expected functionals for stochastic processes}\label{sec:C-space}
In this section, we consider the case where ${\cal X}$ is a  functional space on a compact set $K\subset \R$,  e.g., $\mathbb{C}(K)$, $\mathbb{D}(K)$, with the sup norm 
\[
\|x\| = \sup_{t\in K}|x_t|,\quad x=(x_t)_{t\in K}\in {\cal X}.
\] 
Without loss of generality, we assume that $K=[0,1]$ for notational simplicity, so we consider the case where $X^\th$ is a continuous time stochastic process on $[0,1]$. 

\subsection{Functionals of expected integrals}
In this section, we are interested in the expected integral-type functionals 
\[
H(\th) = \E\l[\int_0^1 V_{\th}(X_t^{\th},t)\,\df t\r],
\]
for a function $V:\R^d\times [0,1]\to \R$. This is the case where $H(\th) = \E[h(X^\th,\th)]$ with 
\[
h(x,\th) = \int_0^1V_\th(x_t,t)\,\df t,\quad x\in {\cal X}
\]
The marginal distribution of a stochastic process $X^\th$ is generally not explicit and the expectation 
$\E[V_\th(X_t^\th,t)]$ is not clear. In such a case, Theorem \ref{thm:cond} can be useful to the analysis if the assumption A4($q$) can be confirmed. 

\begin{example}\label{ex:diff}
Suppose that $X^\th$ satisfies the following 1-dim stochastic differential equation: 
\[
X^\th_t = x(\th) + \int_0^t a(X^\th_s,\th)\,\df s + \int_0^t b(X^\th_s,\th)\,\df W_s, 
\]
where $W$ is a Wiener process and $a,b$ are functions with some ``good" regularities and $\th\in \R^p$ is the unknown parameter. According to Section \ref{sec:sde}, under some regularities, 
the {\it derivative process} $Y^\th = (Y^\th_t)_{t\in [0,1]}$ is given as follows. 
\[
Y_t^\th = \dot{x}(\th) + \int_0^t A(X_s,Y_s,\th)\,\df s + \int_0^t B(X_s,Y_s,\th)\,\df W_s, 
\]
where $\dot{x}(\th) =\n_\th x(\th)$ and $A, B$ are $\R$-valued functions on $\R\times \R^p\times \R$, which are of the form  
\[
A(x,y,\th)= \n_x a(x,\th) y + \dot{a}(x,\th);\quad B(x,y,\th)= \n_xb(x,\th)y + \dot{b}(x,\th). 
\]
This $Y^\th$ can satisfy
\[
\E\|X^{\th+u} - X^\th - u^\top Y^\th \|^p \lesssim |u|^{2p}. 
\]
for each $u\in \R^p$ and any $p\ge 2$, which implies A4($p$). 
\end{example}

\begin{thm}\label{thm:int-V}
Suppose that there exists an integer $n\ge 1$ and $\th\in \Theta$ such that 
$V^{(n)}_\th(x,t):=\n_x^n V_\th(x,t)$ is Lipschitz continuous with respect to $x$, uniformly in $t\in [0,1]$: 
\[
\sup_{t\in [0,1]}|V^{(n)}_\th(x,t) - V^{(n)}_\th(y,t)|\lesssim |x-y|, \quad x,y\in \R. 
\]
Moreover, suppose that A4($q$) holds for some $q\ge 2n$, and that 
\[
\sup_{t\in [0,1]}|V^{(k)}_\th(X^{\th}_t,t)| \in L^r,\quad k=1,\dots,n, 
\]
for some $r>1$ with $1/r + 1/q = 1$. 
Then, condition \eqref{mvt} holds with 
\[
G_{\th}=\int_0^1 V^{(1)}_\th(X^\th_t,t)Y_t^{\th}\,\df t. 
\]
\end{thm}

\begin{proof}
We shall check condition \eqref{mvt} in Theorem \ref{thm:cond}. 
In the proof, for notational simplicity we consider only the case where $d=1$. 
The general case can be shown in a similar manner. 

Let
\[
R_t^{\th,u}:= X_t^{\th + u} - X_t^\th - Y_t^\th, \quad u\in \R^p. 
\]
We note that $\|R^{\th,u}\|_{L^q} \lesssim |u|^q$ by A4($q$). 
It follows from Taylor's formula that
\begin{align*}
&\E\l[h(X^{\th + u},\th) - h(X^\th,\th) - u^\top \int_0^1 V^{(1)}_\th(X^\th_t,t)Y_t^{\th}\,\df t \r] \\
&= E\l[\int_0^1 \l\{V_\th(X_t^{\th+u},t) - V_\th(X_t^\th,t) - u^\top V^{(1)}_\th(X^\th_t,t)Y_t^{\th}\r\}\,\df t \r] \\
&= \E\l[\int_0^1 V^{(1)}_\th(X_t^\th,t) R_t^{\th,u}\,\df t\r] 
+\sum_{k=2}^{n-1}\frac{1}{k!}\E\l[\int_0^1 V^{(k)}_\th(X^\th,t)(R_t^{\th,u} + u^\top Y_t^\th)^k\,\df t\r] \\
&\qquad + \frac{1}{n!}\E\l[\int_0^1 V^{(n)}_\th(\wt{X}_t^{\th,u},t)(R_t^{\th,u} + u^\top Y_t^\th)^n\,\df t\r],  
\end{align*}
where $\wt{X}_t^{u} = X^\th + \eta_\th^u (X^{\th + u} - X^\th)$ for some random number $\eta_\th^u \in [0,1]$.  

Firstly, it follows from H\"older's inequality that for $q,r>1$ with $1/q + 1/r=1$, 
\[
\l|\E\l[\int_0^1 V^{(1)}_\th(X_t^\th,t) R_t^{\th,u}\,\df t\r] \r|
\lesssim \l\|\sup_{t\in [0,1]}V ^{(1)}_\th(X_t^\th,t)\r\|_{L^r} \|R^{\th,u}\|_{L^q}
\]
Secondly, noticing that 
\begin{align*}
\|\wt{X}^{\th,u} - X^\th\|_{L^2} \le \|R^{\th,u}\|_{L^2} + |u| \|Y^\th\|_{L^2} = O(|u|),\quad |u|\to 0, 
\end{align*}
we see that 
\begin{align*}
&\l|\E\l[\int_0^1 V^{(n)}_\th(\wt{X}_t^{u},t)(R_t^{\th,u} + u^\top Y_t^\th)^n\,\df t\r] - \E\l[\int_0^1 V^{(n)}_\th(X_t^\th,t)(R_t^{\th,u} + u^\top Y_t^\th)^n\,\df t\r]\r| \\
&\lesssim 2^{n-1}\int_0^1 \E\l[ |\wt{X}_t^{u} - X_t^\th|\l\{|R_t^{\th,u}|^n + |u|^n |Y_t^\th|^n\r\}\r] \,\df t\\
&\lesssim \|\wt{X}^{\th,u} - X^\th\|_{L^2}\l\{ \| R_\th^u\|_{L^{2n}}^{n} + |u|^n\| Y^\th\|_{L^{2n}}^{n}\r\}
=o(|u|^n),\qquad \mbox{as $|u|\to 0$},  
\end{align*}
Finally, from the Schwartz inequality, it is easy to see that for each $k=2,\dots,n$, 
\begin{align*}
&\l|\E\l[\int_0^1 V^{(k)}_\th(X_t^\th,t)(R_t^{\th,u} + u^\top Y_t^\th)^k\r]\r| \\
&\le \int_0^1 \E\l[|V^{(k)}_\th(X_t^\th,t)|\cdot \|R^{\th,u} + u^\top Y^\th\|^k\r]\,\df t \\
&\le 2^{k-1}\int_0^1 \|V^{(k)}_\th(X_t^\th,t)\|_{L^s}\,\df t \cdot \l\{ \|R^{\th,u}\|_{L^q}^k + |u|^k\|Y^\th\|_{L^q}^k \r\} \\
&= o(|u|^k), 
\end{align*}
where $s>1$ with $1/s + k/q = 1$. Note that such an $s>1$ exists under our assumption 
since $(1 - k/q)^{-1} \ge n/(n-1)> 1$ when $q\ge 2n$. As a result, we have that 
\begin{align*}
\l|\E\l[h(X^{\th + u},\th) - h(X^\th,\th) - u^\top \int_0^1 V^{(1)}_\th(X^\th_t,t)Y_t^{\th}\,\df t \r] \r|
\lesssim \|R^{\th,u}\|_{L^p} + o(|u|^2), 
\end{align*}
which implies condition \eqref{mvt} in Theorem \ref{thm:cond} with 
$G_\th = \int_0^1 V^{(1)}_\th(X^\th_t,t)Y_t^{\th}\,\df t$.  Therefore, the proof is completed. 
\end{proof}

\begin{remark}
If the function $V_\th$ is a ``good" function such that a ``lower" derivative is Lipschitz continuous, 
then Theorem \ref{thm:int-V} requires only a ``small" $q\ge 2$ for A4$(q)$ to hold true. 
The more “violent” the function $V$ is, the stronger the integrability condition becomes.
\end{remark}

\begin{example}
Consider a 1-dim (ergodic) diffusion process $X^\th=(X_t)_{t\ge 0}$:  for a constant $x>0$,  
\[
X^{\th_0}_t = x + \int_0^t a(X^{\th_0}_s,\th_0)\,\df s + \int_0^t b(X^{\th_0}_t)\,\df W_t,  
\]
where $\th_0\in \R$ is unknown, and consider the estimation of
\[
H(\th_0) =  \int_0^T e^{-r t}U(X^{\th_0}_t)\,\df t, 
\]
for a constant  $r>0$ and a function $U\in C(\R)$, 
which is the case where $V_\th(x,t) = e^{-r t}U(x)\I_{[0,T)}(t)$. See also Example \ref{ex:ou2} for practical applications of this example. 

Assume that we have continuous data $\{X_t\}_{t\in [0,T]}$, and consider the {\it long term asymptotics}: $T\to \infty$. Then, under some regularities, the maximum likelihood estimator of $\th$, say $\wh{\th}_T$, satisfies
\[
\sqrt{T}(\wh{\th} - \th_0) \toD N(0,I^{-1}(\th_0)), \quad T\to \infty, 
\]
where $I(\th) = \int_{\R} \frac{a^2(x,\th)}{b^2(x)}\,\pi(\df x)$ for a stationary distribution $\pi$, 
and it can be estimated by, e.g., 
\[
\wh{I}_T(\th) = \frac{1}{T}\int_0^T \frac{a^2(X_t,\th)}{b^2(X_t)}\,\df t \toP I(\th), \quad T\to \infty, 
\]
uniformly in $\th\in \Theta$ (see, e.g., Kutoyants \cite{k04}). 
Therefore, considering the derivative process $Y^\th$ given in Example \ref{ex:diff}, we have that 
\[
\sqrt{T}(H(\wh{\th}_T) - H(\th_0)) \toD N(0, C_{\th_0}^2 I(\th_0)^{-1}), \quad T\to \infty,  
\]
where 
\[
C_\th = \E\l[\int_0^T e^{-r t}\n_x U(X_t^\th)Y^\th_t\,\df t\r]
\]
Therefore an $\a$-confidence interval for $H(\th_0)$ is given by 
\[
\l[H(\wh{\th}_T) - \frac{z_{\a/2}}{\sqrt{T}} C_{\wh{\th}_T}\wh{I}_T(\wh{\th}_T)^{-1/2}, 
H(\wh{\th}_T) + \frac{z_{\a/2}}{\sqrt{T}} C_{\wh{\th}_T}\wh{I}_T(\wh{\th}_T)^{-1/2} \r]. 
\]
In practice, $H(\wh{\th}_T)$ and $C_{\wh{\th}_T}$ will be computed by Monte Carlo simulations by a suitable discretization if needed. Of course, the same argument is possible in the case where $X^\th$ is discretely observed; cf. Example \ref{ex:ou2}. 
\end{example}

\subsection{Functionals of integrated professes}
Let us consider the following quantity: for a function $\vp_\th:\R\to \R$ and $T\in (0,1]$, 
\[
H(\th) = \E\l[\vp_{\th}\l(\frac{1}{T}\int_0^T X_t^{\th} \,\df t\r)\r]. 
\]
We use the following notation for simplicity:  
\[
X_* = \frac{1}{T}\int_0^T X_t\,\df t,  
\]
for a process $X=(X_t)_{t\in [0,1]}$. Then, we have the following theorem. 
\begin{thm}\label{thm:asian}
Suppose that there exists an integer $n\ge 1$ and $\th\in \Theta$ such that 
$\vp_{\th}^{(n)}(x)$ is Lipschitz continuous: 
\[
|\vp_{\th}^{(n)}(x) - \vp_{\th}^{(n)}(y)|\lesssim |x-y|, \quad x,y\in \R. 
\]
Moreover, suppose that A4($q$) holds for some $q\ge 2n$, and that 
\[
\vp_{\th}^{(k)}(X^{\th}_*) \in L^r,\quad k=1,\dots,n, 
\]
for the constant $r>1$ with $1/r + 1/q = 1$. 
Then, condition \eqref{mvt} holds with 
\[
G_\th = \vp^{(1)}_\th (X^\th_*)Y^\th_*. 
\]
\end{thm}

\begin{proof}
It follows from Jensen's inequality that   
\begin{align*}
|X^{\th+u}_* - X^\th_* - u^\top Y^\th_*| &\le \frac{1}{T}\int_0^T |X^{\th+u}_t - X^\th_t - u^\top Y^\th_t|\,\df t  \le \|X^{\th + u} - X^\th - u^\top Y^\th\|, 
\end{align*}
with probability one. 
Hence, $Y^\th_* = \frac{1}{T}\int_0^T Y^\th_t\,\df t$ is the derivative of $X^\th$ w.r.t. $\th$.  

We can take the same argument as in Theorem \ref{thm:int-V}: we use Taylor's formula and  H\"older's inequality to obtain
\begin{align*}
&\l|\E[h(X^{\th+u}) - h(X^\th) -  \vp^{(1)}(X^\th_*)u^\top Y^\th_*]\r|  \\
&\le \E\l| \vp^{(1)}(X^{\th}_*)(X^{\th+u}_* - X^\th_* - u^\top Y^\th_*)\r| 
+ \sum_{k=1}^{n-1}\frac{1}{k!}\E\l| \vp^{(k)}(X^{\th}_*)(X^{\th+u}_* - X^\th_*)^k \r| \\
&\quad + \frac{1}{n!} \E\l| \vp^{(k)}(\wt{X}^{\th,u}_*)(X^{\th+u}_* - X^\th_*)^n \r|. 
\end{align*}
Then, the same argument as in the proof of Theorem \ref{thm:int-V} 
enables us to check condition \eqref{mvt} in Theorem \ref{thm:cond}. 

\end{proof}

\begin{example}\label{ex:asian}
When $X^\th$ is a stock price, the price of an {\it Asian call option} for $X^\th$ with maturity $T$ and strike price $K$ is given by 
\[
{\cal C}_{T,K} = \E\l[\max \l\{ \frac{1}{T}\int_0^T X_t^{\th_0}\,\df t - K,0\r\} \r]
\]
where $\d>0$ is an interest rate and $\E$ is usually taken as an expectation with respect to the risk-neutral probability. This is approximated as 
\[
H_\e(\th_0) := \E\l[\vp_\e\l(  \frac{1}{T}\int_0^T X_t^{\th_0}\,\df t \r)\r]
\]
by a function $\vp_\e(x) \in C^\infty(\R)$ such that 
\[
\sup_{x}|\vp_\e(x) - \max\{x-K,0\}| \to 0,\quad \e\to 0. 
\]
For example, we can take a function $\vp_\e(x) = 2^{-1}(\sqrt{(x-K)^2 + \e^2} + x - K)$. 
Then, it follows by the dominated convergence theorem that $H_\e(\th) \to {\cal C}_{T,K}$ as $\e\to 0$ if $X^\th\in L^1$. 

Assume that a suitable estimator of $\th_0\in \R^p$ is obtained, e.g., 
\[
\sqrt{T} (\wh{\th}_T - \th_0) \toD N(0,\Sigma),\quad T\to \infty, 
\]
for a positive-definite matrix $\Sigma\in \R^p\otimes \R^p$. 
Then, we can apply Theorem \ref{thm:asian} to $H_\e(\th)$, and letting $T\to \infty$ as well as $\e\to 0$,  we have 
\[
\sqrt{T}\l(H_\e(\wh{\th}_T) - H_\e(\th_0)\r) \toD  N(0,C_{\th_0}^\top \Sigma \,C_{\th_0}), 
\]
where 
\[
C_{\th} = \lim_{\e\to 0}\E\l[\frac{Y^\th_*}{2}\l\{\frac{(X^\th_* - K)}{\sqrt{(X^\th_* - K)^2 + \e^2}} + 1\r\}\r]  = \E\l[\frac{Y^\th_*}{2}\l\{\mathrm{sgn}(X^\th_* - K) + 1\r\} \r], 
\]
with $\mathrm{sgn}(z) = \I_{\{z>0\}} - \I_{\{z<0\}}$.  
Note that this quantity would be computed by Monte Carlo simulation in practice with $\th_0$ replaced by $\wh{\th}_T$, or some estimators based on discrete samples of $X^\th$ in practice. 
We will discuss when the condition A4$(q)$ holds when $X^\th$ is a semimartingale with jumps in Section \ref{sec:sde}. 
\end{example}

\begin{remark}\label{rem:general}
According to the proof of Theorem \ref{thm:asian}, we can consider more general functionals for $X^\th_*$ under some smoothness conditions for $\vp_\th$. 
That is, suppose that there exists an $\R^p$-valued random variable $\wt{Y}^\th$ such that the following inequality holds:
\begin{align}
|X^{\th+u}_* - X^\th_* - u^\top \wt{Y}^\th|\lesssim \|X^{\th+u} - X^\th - u^\top Y^\th\| + |u|^{1+\d}\quad a.s., \label{eq:X-star}
\end{align}
for $\d>0$, and the derivative is $Y^\th$. Then, the same proof as that of Theorem \ref{thm:asian} works with 
\begin{align*}
G_\th = \vp^{(1)}_\th(X^\th_*) \wt{Y}^\th. 
\end{align*}
For example, let  
\[
X^\th_* = \int_0^T U(X^\th_t)\,\df t
\]
for $T>0$ and $U\in C^2(\R)$ be a function with bounded derivatives. Then we find that
\[
\wt{Y}^\th = \int_0^T U^{(1)}(X^\th_t)Y^\th_t\,\df t, 
\] 
since it follows that 
\begin{align*}
|X^{\th+u}_* - X^\th_*  - u^\top \wt{Y}^\th| &\le \int_0^T |U(X^{\th + u}_t) - U(X^\th_t) - u^\top U^{(1)}(X^\th_t)Y^\th_t|\,\df t \\
&\lesssim \int_0^T |U^{(1)}(X^\th_t)(X^{\th+u}_t - X^\th_t - u^\top Y^\th_t)\,\df t + |u|^2 \\
&\lesssim \|X^{\th+u} - X^\th - u^\top Y^\th\| + |u|^2. 
\end{align*}
This argument can include Theorem \ref{thm:int-V}. 
\end{remark}

\begin{remark}
You might also be interested in the case where $X^\th_*$ is an extreme-type functional such as 
$X_*^\th = \inf_{s \le t} X^\th_s$, which is important when, e.g., $\vp(x) = \I_{\{x < 0\}}$, the function $H(\th) =\E[\vp(X_*^\th)]$ stands for the hitting time distribution: 
\[
H(\th) = \P(\tau^\th \le t),\quad \tau^\th = \inf\{t>0\,|\,X^\th_t < 0\}, 
\]
or we can approximate $\vp$ with a bounded smooth function such as, e.g.,  $\vp_\e(x) = [1 + e^{-x/\e}]^{-1}\ \to \vp(x)\ (\e\to 0)$, 
among others. 

When $X^\th$ is a continuous diffusion process, Gobet and Kohatsu \cite{gk03} obtain a derivative of $\P(\tau^\th \le t)$ via Malliavin Calculus. 
However, in our approach, it is not so easy to find a suitable random variable $\wt{Y}^\th$ satisfying the inequality \eqref{eq:X-star}, except for a trivial case where the derivative process $Y^\th$ is a constant. 
One might expect that $\wt{Y}^\th = \sup_{s\le t}Y_s^\th$ in general, but it fails. 
This important  case is an open problem.  
\end{remark}

\section{Expected functionals of semimartingales}\label{sec:sde}

\subsection{Stochastic differential equations with jumps}

On a stochastic basis $(\Omega,\F,\mathbb{F},\P)$ with a filtration $\mathbb{F}=(\F_t)_{t\ge 0}$, 
consider a $1$-dim stochastic process $X=(X_t)_{t\in [0,T]}$ that satisfies the following stochastic differential equation (SDE) with a  multidimensional parameter $\th\in \Theta\subset \R^p$:  
\begin{align}
X^\th_t = x(\th) + \int_0^t a(X^\th_s,\th)\,\df s + \int_0^t b(X^\th_s,\th)\,\df W_s 
+\int_0^t \int_{E} c(X^\th_{s-},z,\th)\,\wt{N}(\df t,\df z),  \label{sde}
\end{align}
where $E=\R \setminus \{0\}$; $x:\Theta \to \R$; $a:\R\times \Theta\to \R$, $b:\R \times \Theta \to \R\otimes \R$ and $c:\R\times E\times \Theta\to \R$; $W$ is a $\mathbb{F}$-Wiener process. 
Moreover, $\wt{N}(\df t,\df z):=N(\df t,\df z) - \nu(z)\,\df z\df t$, which is the {\it compensated Poisson random measure}, where $N$ is a Poisson random measure associated with a $\mathbb{F}$-L\'evy process, say $Z=(Z_t)_{t\ge 0}$ with the L\'evy density $\nu$: 
\[
N(A\times (0,t]) = \sum_{s \le t}\I_{\{\D Z_s\in A\}},\quad A\subset E,  
\]
and $\E[N(\df t,\df z)] = \nu(z)\,\df z\df t$. 

In what follows, we assume that $\nu$ is {\it essentially} known:   
some cases can be rewritten into a model for a known $\nu$ even if $\nu$ has some unknown parameters (see Remark \ref{rem:unknown} below). However, if it is not the case, the situation may be totally different from ours, and the argument in this section would no longer work; see Remark \ref{rem:unknown2}.

\begin{remark}\label{rem:unknown}
Some cases where the L\'evy density $\nu$ depends on an unknown parameter, say $\nu_\th$, can be rewritten into the form of \eqref{sde} with a known L\'evy process by changing the coefficients $a$ and $c$, suitably.   
For example, consider the following SDE: 
\begin{align}
\df X_t = a(X_t)\,\df t + b(X_t)\,\df W_t +  \int_E c(X_{t-},z)\,N_\th(\df t,\df z), \label{sde2}
\end{align}
where $N_\th$ is the Poisson random measure associated with a compound Poisson process of the form 
$Z^\th_t = \sum_{i=1}^{N_t} U^\th_i$ such that $N$ is a Poisson process with intensity $\la_0$, and the $U_i^\th$'s are i.i.d. sequences with probability density $f_\th$ with $\E[U^\th_i] = \eta$ and $Var(U^\th_i) = \z^2$. 
Suppose that $\la_0$ is known, but $\th=(\eta,\z)$ is unknown. 
In this case, we can rewrite $Z^\th\,(=Z^{(\eta,\z)})$ as 
\[
Z^{(\eta,\z)}_t = \sum_{i=1}^{N_t} (\z U_i^{(0,1)} + \eta) = \int_0^t \int_E (\z z + \eta)\,N_{(0,1)}(\df s,\df z),  
\]
where $N_{(0,1)}$ is the Poisson random measure associated with $Z^{(0,1)}$. 
Then, the SDE \eqref{sde2} is written as 
\begin{align*}
\df X_t &= a(X_t)\,\df t + b(X_t)\,\df W_t +  \int_E c(X_{t-},\z z+\eta)\,N_{(0,1)}(\df t,\df z) \\
&= \l[a(X_t) +  \la_0\int_E c(X_t,\z z+\eta)f_{(0,1)}(z)\,\df z\r]\,\df t+ b(X_t)\,\df W_t  \\
&\qquad +  \int_E c(X_{t-}, \z z+\eta)\,\wt{N}_{(0,1)}(\df t,\df z),   
\end{align*}
where the L\'evy density $\la_0 f_{(0,1)}(z)$ is known. See also Example \ref{ex:ou}. 
\end{remark}

The semimartingale $X^\th$ in \eqref{sde} is a ${\cal X}=\mathbb{D}([0,T])$-valued random element. 
In what follows, we consider a metric space $({\cal X},\|\cdot\|)$  with the sup norm: 
\[
\|X^\th\|=\|X^\th\|_T := \sup_{t\in [0,T]}|X_t^\th|. 
\] 

We make some assumptions. 

\begin{asb}\label{asb:lg}
For each $x, z\in \R$,  
\[
|a(x,\th)| + |b(x,\th)|  \lesssim 1 + |x|;\quad |c(x,z,\th)|\lesssim |z|(1 + |x|), 
\]
uniformly in $\th\in \Theta$. 
\end{asb}

\begin{asb}\label{asb:smooth}
The functions $a,b$ and $c$ are twice differentiable in $x$, and 
the derivatives $\n_x^k a$ and $\n_x^k b$ $(k=1,2)$ are uniformly bounded. 
Moreover, $|\n_x^k c(x,z,\th)| \lesssim |z|$. 
\end{asb}

\begin{asb}\label{asb:dot}
The functions $a,b$ and $c$ are differentiable in $\th$.
It follows that 
\[
|\dot{a}(x,\th)| + |\dot{b}(x,\th)|\lesssim 1 + |x|; \quad |\dot{c}(x,z,\th)|\lesssim  |z|(1 + |x|), 
\]
uniformly in $\th\in \Theta$. 
\end{asb}

\begin{asb}\label{asb:nu-moment}
For any $p>0$, $\dis \int_{|z|>1} z^p\,\nu(z)\,\df z < \infty.$ 
\end{asb}

\begin{asb}\label{asb:moment}
For any $p>0$ and $T>0$, $\|X\|_T^p < \infty$. 
\end{asb}

\subsection{Derivative processes}

Let $Y^\th =(Y^\th_t)_{t\ge 0}$ be a $p$-dim stochastic process satisfying the following SDE:  $Y^\th_0=\dot{x}(\th)$, 
\begin{align}
\df Y^\th_t &= A(X^\th_t,Y^\th_t,\th)\,\df t + B(X^\th_t,Y^\th_t,\th)\,\df W_t + \int_{E} C(X^\th_{t-},Y^\th_{t-},z,\th)\,\wt{N}(\df t,\df z), \label{Y-sde} 
\end{align}
for each $\th\in \Theta$, where 
\begin{align*}
A(x,y,\th) &= \n_x a(x,\th)y + \dot{a}(x,\th); \\
B(x,y,\th) &= \n_x b(x,\th)y + \dot{b}(x,\th); \\
C(x,y,z,\th) &= \n_x c(x,z,\th)y + \dot{c}(x,z,\th). 
\end{align*}

In this section, we will show that the above $Y^\th=(Y^\th_t)_{t\ge 0}$ can be 
the {\it derivative process} of $X^\th$ with respect to $\th$ in the sense of $L^q$. 
For that purpose, we shall give some preliminary lemmas.

\begin{lemma}\label{lem:M}
Let $g:\R\times \R\to \R$ be of polynomial growth. 
Then ,under B\ref{asb:moment}, it holds for $p=2^m\,(m\in \N)$ that 
\[
\E\l\| \int_0^t \int_E g(X_{s-},z)\,\wt{N}(\df s,\df z) \r\|_T^p \lesssim \E\l[\int_0^T\int_E |g(X_{s-},z)|^p\,\nu(z)\,\df z\df s\r]
\]
\end{lemma}

\begin{proof}
See Shimizu and Yoshida \cite{sy06}, Lemma 4.1.
\end{proof}

\begin{lemma}\label{lem:DX}
Suppose that assumptions B\ref{asb:lg} -- B\ref{asb:moment} hold, and that $\dot{x}(\th)$ is uniformly bounded on $\Theta$. 
Then, it follows for any $T>0$, $p\ge 2$ and $u\in \R^p$ with $\th+u\in \Theta$ that
\[
\E\|X^{\th+u} - X^\th\|_T^p \lesssim |u|^p. 
\]
\end{lemma}

\begin{proof}
It follows from Jensen's inequality that 
\begin{align}
|X^{\th+u}_t - X^\th_t|^p 
&\lesssim |x(\th+u) - x(\th)|^p + t^{p-1}\int_0^t |\wt{A}_s(u,\th)|^p\,\df s + \l|\int_0^t \wt{B}_t(u,\th)\,\df W_s\r|^p \notag \\
&\quad + \l|\int_0^t \int_{E}\wt{C}_s(u,z,\th)\,\wt{N}(\df s,\df z)\r|^p, \label{DX}
\end{align}
with 
\begin{align*}
\wt{A}_t(u,\th)&:= a(X^{\th+u}_t,\th+u) - a(X^{\th}_t,\th); \\
\wt{B}_t(u,\th)&:= b(X^{\th+u}_t,\th+u) - b(X^{\th}_t,\th); \\
\wt{C}_t(u,z,\th)&:= c(X^{\th+u}_t,z,\th+u) - c(X^{\th}_t,z,\th). 
\end{align*}
Then, since it holds that $|x(\th+u) - x(\th)|\lesssim |u|$ from the mean value theorem, 
Lemma \ref{lem:M} and Burkholder-Davis-Gundy's inequality yield  that 
\begin{align*}
\E\l\|X^{\th+u} - X^\th \r\|_T^p &\lesssim |u|^{p} + \int_0^T \E\l[ |\wt{A}_s(u,\th)|^p + |\wt{B}_s(u,\th)|^p\r]\,\df s +  \E\l[\int_0^T \int_{E}|\wt{C}_s(u,z,\th)|^p\,\nu(z)\,\df z\df s\r] 
\end{align*}
It follows from the mean value theorem and assumptions B\ref{asb:lg} -- B\ref{asb:dot} that 
\begin{align*}
|\wt{A}_t(u,\th)|^p &= |\n_xa(X^*,\th^*)(X_t^{\th+u} - X_t^\th) + \dot{a}(X^*,\th^*)^\top u|^p \\
&\lesssim |X_t^{\th+u} - X_t^\th|^p + \l(1 + \|X\|_T^p\r)|u|^p
\end{align*}
Hence, it follows from B\ref{asb:moment} that 
\[
\E|\wt{A}_t(u,\th)|^p \lesssim \E|X_t^{\th+u} - X_t^\th|^p + |u|^p
\]

Similarly, we also have that 
\begin{align*}
\E|\wt{B}_t(u,\th)|^p &\lesssim \E|X^{\th+u}_t - X^\th_t|^p +  |u|^p;\\ 
\E|\wt{C}_t(u,z,\th)|^p &\lesssim|z|^p\l(\E|X^{\th+u}_t - X^\th_t|^p +  |u|^p\r).
\end{align*}
Hence, assumption B\ref{asb:nu-moment} yields that 
\[
\E\l\|X^{\th+u} - X^\th \r\|_T^p \lesssim |u|^{p} + \int_0^T \E\l\|X^{\th+u} - X^\th \r\|_t^p \,\df t. 
\]
Finally, Gronwall's inequality completes the proof.
\end{proof}

The next theorem is the consequence of this section. 

\begin{thm}\label{thm:deriv-X}
Suppose that assumptions B\ref{asb:lg} -- B\ref{asb:moment} hold. 
Moreover, suppose that the initial value $x(\th)=X^\th_0$ is twice differentiable with respect to the bounded derivatives, and that 
the solution $Y^\th$ to \eqref{Y-sde} satisfies $\|Y^\th\|_T < \infty$ for any $T>0$. 
Then, for any $p\ge 2$, there exists a positive constant $C_p$ depending on $p$ such that 
\[
\E\l\|X^{\th+u} - X^\th - u^\top Y^\th \r\|_T^p \le C_p|u|^{2p},\quad h \in \R^p.  
\]
\end{thm}

\begin{proof}
First, we shall consider the case where $p=2^m\ (m\in \N)$. 
Applying Jensen's inequality to the $\df t$-integral part, we see that 
\begin{align}
|X^{\th+u}_t - X^\th_t - u^\top Y^\th_t|^p 
&\lesssim |x(\th+u) - x(\th) - u^\top \dot{x}(\th)|^p + t^{p-1}\int_0^t |\wt{A}_s(u,\th)|^p\,\df s \notag \\
&\quad + \l|\int_0^t \wt{B}_t(u,\th)\,\df W_s\r|^p + \l|\int_0^t \int_{E}\wt{C}_{s-}(u,z,\th)\,\wt{N}(\df s,\df z)\r|^p, \label{DX-Y}
\end{align}
where 
\begin{align*}
\wt{A}_t(u,\th)&:= a(X^{\th+u}_t,\th+u) - a(X^{\th}_t,\th) - u^\top [\n_x a(X^\th_t,\th)Y^\th_t+ \dot{a}(X^\th_t,\th)]; \\
\wt{B}_t(u,\th)&:= b(X^{\th+u}_t,\th+u) - b(X^{\th}_t,\th) - u^\top [\n_x b(X^\th_t,\th)Y^\th_t+ \dot{b}(X^\th_t,\th)]; \\
\wt{C}_t(u,z,\th)&:= c(X^{\th+u}_t,z,\th+u) - c(X^{\th}_t,z,\th) - u^\top [\n_x c(X^\th_t,z,\th)Y^\th_t+ \dot{c}(X^\th_t,z,\th)]. 
\end{align*}

Take $\sup_{t\in [0,T]}$ and the expectation $\E$ on both sides to obtain that 
\begin{align*}
\E\|X^{\th+u} - X^\th - u^\top Y^\th\|_T^p 
&\lesssim |u|^{2p} + \int_0^T \E|\wt{A}_t(u,\th)|^p\,\df t + \E\l\|\int_0^\cdot \wt{B}_s(u,\th)\,\df W_s\r\|_T^p\\
&\quad + \E\l\|\int_0^\cdot \int_{E}\wt{C}_{s-}(u,z,\th)\,\wt{N}(\df s,\df z)\r\|_T^p. 
\end{align*}
Using Burkholder-Davis-Gundy's inequality and Lemma \ref{lem:M}, we have that 
\begin{align*}
\E\|X^{\th+u} - X^\th - u^\top Y^\th\|_T^p 
&\lesssim |u|^{2p} + \int_0^T \E|\wt{A}_s(u,\th)|^p\,\df s + \E\l|\int_0^T \wt{B}_s^2(u,\th)\,\df s\r|^{p/2} \\
&\quad +  \E\l[\int_0^T \int_{E}\wt{C}_s^p(u,z,\th)\,\nu(z)\,\df z\df s\r] \\
&\lesssim |u|^{2p} + \int_0^T \E\l[ |\wt{A}_s(u,\th)|^p + |\wt{B}_s(u,\th)|^p\r]\,\df s \\
&\quad +  \E\l[\int_0^T \int_{E}|\wt{C}_s(u,z,\th)|^p\,\nu(z)\,\df z\df s\r] 
\end{align*}

According to assumptions B\ref{asb:lg}, B\ref{asb:smooth}, and Taylor's formula, we have, e.g., 
\begin{align*}
\wt{A}_t(u,\th) &= \n_x a(X^{\th}_t)(X^{\th+u}_t - X^\th_t) + u^\top \dot{a}(X^\th,\th)   + \frac{1}{2}[(X^{\th+u}-X^\th)\n_x + u^\top \n_\th]^2a(X_t^*,\th^*) \\
&\quad - u^\top [\n_x a(X^\th_t,\th)Y^\th_t+ \dot{a}(X^\th,\th)],  
\end{align*}
where $X^*$ is a random variable between $X_t^{\th+u}$ and $X^\th$, $\th^* \in [\th,\th+u]$.  
Since the second derivatives are bounded, and from B\ref{asb:dot}, we have that 
\begin{align*}
|\wt{A}_t(u,\th)|^p &\lesssim|X^{\th+u}_t - X^\th_t - u^\top Y^\th_t|^p + |u|^{2p} + |X^{\th+u}_t-X^\th_t|^{2p} 
+ |u|^p|X^{\th+u}_t-X^\th_t|^p
\end{align*}
Similarly, we also have that 
\begin{align*}
|\wt{B}_t(u,\th)|^p &\lesssim|X^{\th+u}_t - X^\th_t - u^\top Y^\th_t|^p + |u|^{2p} + |X^{\th+u}_t-X^\th_t|^{2p} 
+ |u|^p|X^{\th+u}_t-X^\th_t|^p;\\ 
|\wt{C}_t(u,z,\th)|^p &\lesssim|z|^p\l(|X^{\th+u}_t - X^\th_t - u^\top Y^\th_t|^p + |u|^{2p} + |X^{\th+u}_t-X^\th_t|^{2p} + |u|^p|X^{\th+u}_t-X^\th_t|^p\r).
\end{align*}

Hence, under B\ref{asb:nu-moment}, it follows from Lemma \ref{lem:DX} that 
\[
\E\|X^{\th+u} - X^\th - u^\top Y^\th\|_T^p 
\lesssim |u|^{2p} + \int_0^T \E\|X^{\th+u} - X^\th - u^\top Y^\th\|_t^p \,\df t, 
\]
and Gronwall's inequality yields the consequence. 

For any $p\ge 2$, we write the binomial expansion of $p$ as $p= \sum_{k=1}^m \d_k 2^k$, where $m$ is an integer and $\d_k = 0$ or $1$.  
Note that we have already proved the consequence for $p$ with $m=1$ and $\d_1=0,1$. 
Next, we assume that the consequence also  holds true for some $m$ and any $\d_k\ (k=1,2,\dots,m)$. 
Then, the Cauchy-Schwartz inequality yields that for $q = \sum_{k=2}^{m} 2^{k}\d_{k-1}$, 
\begin{align*}
&\E\|X^{\th+u} - X^\th - u^\top Y^\th \|_T^p \\
&= \E\l[  \|X^{\th+u} - X^\th - u^\top Y^\th \|_T^{2^m\d_m}\prod_{k=1}^{m-1} \|X^{\th+u} - X^\th - u^\top Y^\th \|_T^{\d_k2^k}\r] \\
&\le \sqrt{\E\l[ \l\|X^{\th+u} - X^\th - u^\top Y^\th \r\|_T^{2^{m+1}\d_m}\r] } 
\sqrt{\E\l[  \l\|X^{\th+u} - X^\th - u^\top Y^\th \r\|_T^{\sum_{k=2}^{m} 2^{k}\d_{k-1}}\r] } \\
&\le \sqrt{C_{2^m\d_m}|u|^{2\cdot 2^{m+1} \d_m}} \sqrt{C_q|u|^{2q}} \\
&\lesssim |u|^{2^{m+1}\d_m + q} = |u|^{2p}. 
\end{align*}
This completes the proof. 
\end{proof}

\begin{remark}\label{rem:unknown2}
If the random measure $N$ essentially includes unknown parameters, then the derivative process in the sense of $L^q$ cannot exist. To see this, consider a simple case where $X^\th$ is a Poisson process with (unknown) intensity $\th$: $X^\th\sim Po(\th t)$, which is not the case described in Remark \ref{rem:unknown}. In this case, we cannot compute the expectation $\E\|X^{\th+u} - X^\th \|_T^p$ since we do not know the joint distribution of $(X^{\th+u},X^\th)$. 
 This consideration indicates that we should be careful when we compute expected functionals of $X^\th$ by Monte Carlo simulation when it has an unknown jump part. 
\end{remark}

\begin{example}[L\'evy processes]\label{ex:levy}
Consider a 1-dim L\'evy process $X^\th$ starting at $x>0$: 
\[
X^\th_t = x + \mu t + \s W_t + \eta S_t, 
\]
where $S$ is a known {\it L\'evy process} with $\E[S_1]=1$ and $\eta\ne 0$. 
We set $\th=(\mu,\s,\eta) \in \Theta\subset \R^3$. 
Then, this is the case of \eqref{sde} with 
\[
a(x,\th) = \mu + \eta,\quad b(x,\th) = \s,\quad c(x,z,\th)=\eta z,\quad X_0=x, 
\]
Hence, the derivative process $Y^\th$ is a 3-dim  L\'evy process of the form 
\[
Y^\th_t =\l(t, W_t, S_t\r)
\]
\end{example}

\begin{example}\label{ex:ou}
Consider an  O-U process $X=(X_t)_{t\ge 0}$ written as 
\begin{align}
\df X^\th_t = -\mu X^\th_t\,\df t + \s \,\df W_t + \df Z^\eta_t, \quad X_0 = x\ (\mbox{const.}) \label{ou}
\end{align}
where $\th = (\mu,\s,\eta)$, $W$ is a Wiener process, and $Z^\eta$ is a compound Poisson process with 
known intensity, and  the mean of the jumps  is $\eta$
Then, the SDE \eqref{ou} is rewritten as 
\[
X^\th_t = x +  \int_0^t (-\mu X^\th_s + \eta)\,\df s + \s W_t + \int_0^t \int_E (z+\eta)\wt{N}(\df t,\df z),
\]
where $\wt{N}$ is the compensated Poisson random measure associated with $Z^0\,(\eta =0)$ (see Remark \ref{rem:unknown}). 

Then, the derivative process $Y^\th=(Y^1_t,Y^2_t,Y^3_t)_{t\ge 0}$ satisfies the following SDE: 
\begin{align*} 
Y^1_t &= \int_0^t (-\mu Y_s^1 - X^\th_s)\,\df s; \quad
Y^2_t = - \mu \int_0^t Y^2_s\,\df s + W_t; \\
Y^3_t &=  \int_0^t (1 - \mu Y^3_s)\,\df s + Z_t^0, 
\end{align*}
since $\int_E z\nu_0(z)\,\df z = 0$. 
The equation for $Y^1$ is an ordinary differential equation for almost all $\omega \in \Omega$, and 
the equations for $Y^2$ and $Y^3$ are O-U type SDEs. 
Therefore, we can solve these equations explicitly, as follows: 
\begin{align*}
Y^1_t &= -\int_0^t X^\th_s e^{-\mu(t-s)}\,\df s,\quad 
Y^2_t = \int_0^t e^{-\mu(t-s)}\,\df W_s,\\
Y^3_t &= \frac{1}{\mu} (1-e^{-\mu t}) + \int_0^t e^{-\mu (t-s)}\,\df Z^0_s,
\end{align*}
and 
\begin{align}
X^\th_t = xe^{-\mu t} + \int_0^t e^{-\mu(t-s)}[\s\,\df W_s + \df Z^\eta_s]. \label{eq:ou} 
\end{align}
\end{example}

\subsection{Expected functionals for semimartingales}

For each $\th\in \Theta$, let $\vp_\th:\R\to \R$ and 
\[
H(\th) = \E\l[\vp_\th(X^{\th}_*)\r], 
\]
where $X^\th_*$ is a $\R$-valued random functional of $X^\th$ such that the inequality 
\begin{align}
|X^{\th+u}_* - X^\th_* - u^\top \wt{Y}^\th|\lesssim \|X^{\th+u} - X^\th - u^\top Y^\th\| + |u|^{1+\d}\quad a.s.,  \label{Y-tilde}
\end{align}
holds true for some $\wt{Y}^\th$ and $\d>0$; see Remark \ref{rem:general} for some examples. 
Summing up our results in Sections \ref{sec:mc}, \ref{sec:C-space} and \ref{sec:sde} with Remark \ref{rem:general}, we can immediately obtain the following result. 
\begin{thm}\label{thm:final}
Suppose that the same assumptions as in Theorem \ref{thm:deriv-X} hold. 
Moreover, suppose that there exists an integer $n\ge 1$  such that 
$\vp_{\th_0}^{(n)}(x)$ is Lipschitz continuous: 
\[
|\vp_{\th_0}^{(n)}(x) - \vp_{\th_0}^{(n)}(y)|\lesssim |x-y|, \quad x,y\in \R, 
\]
and that for some constant $r>2$,   
\[
\vp_{\th_0}^{(k)}(X^{\th_0}_*) \in L^r,\quad k=1,\dots, n. 
\]
Furthermore, assume that we have an estimator of $\th_0$ based on some observations depending on a parameter $n$, say $\wh{\th}_n$, such that assumption A\ref{as:Z} holds true. 
Then the asymptotic distribution of $H(\wh{\th}_n)$ is specified: 
\[
\g_{n*}^{-1}(H(\wh{\th}_n) - H(\th_0)) \toD \l(\E\l[\dot{\vp}_{\th_0}(X^{\th_0})\r] + C_{\th_0}\r)^\top Z^*,\quad n\to \infty, 
\]
and the deterministic vector $C_\th$ is given by 
\[
C_\th = \E\l[\vp^{(1)}(X^\th_*)\wt{Y}^\th\r], 
\]
where $\wt{Y}^\th$ is given in \eqref{Y-tilde}. 
\end{thm}

\begin{example}[Ornstein-Uhlenbeck type processes]\label{ex:ou2}
This is a continuation of the previous Example \ref{ex:ou}. 
Let us consider the same SDE as \eqref{ou}, and  consider the {\it expected discounted functional} for a constant $\d>0$, 
\[
H(\th) = \E\l[\int_0^T e^{-\d t} V(X_t) \,\df t\Big| X_0=x\r], 
\]
which is an important quantity in insurance and finance because such a functional can represent an option price when $X$ is a stock price (see, e.g., Karatzas and Shereve \cite{ks91}), or it can represent some aggregated “costs” or “risks” in insurance businesses when $X$ is an asset process of the company; see, e.g., Feng and Shimizu \cite{fs13}. The constant $\d>0$ is interpreted as an interest rate. 

Here, we shall consider a simple case where $V(x) = x$: 
\[
H(\th) = \int_0^t e^{-\d t}\E[X_t]\,\df t, 
\]
Noticing that from expression  \eqref{eq:ou}, 
\[
\E[X_t] = xe^{-\mu t} + \E\l[\int_0^t e^{-\mu(t-s)} \,\df Z^\eta_s\r] = xe^{-\mu t} + \frac{\eta}{\mu}(1-e^{-\mu t}), 
\]
we can compute $H(\th)$ explicitly as 
\[
H(\th) = \frac{x}{\mu + \d}(1-e^{-(\mu + \d)T}) 
+ \frac{\eta}{\mu}\l[\frac{1}{\d} (1-e^{-\d T}) - \frac{1}{\mu + \d}(1-e^{-\mu + \d)T})\r]. 
\]

Suppose that $Z$ is a compound Poisson process, and that we have a set of discrete samples 
$(X_{t_1},X_{t_2},\dots,X_{t_n})$ with $t_k = kh_n$ for $h_n>0$, and 
assume some asymptotic conditions on $n$ and $h_n$, e.g., $h_n\to 0$ and $nh_n^2\to 0$. 
Although we omit the details of the regularity conditions here, we can construct an asymptotic normal (efficient) estimator of $\th = (\mu,\s,\eta)$, say $\wh{\th}_n$, such that 
\[
\G_n^{-1} (\wh{\th}_n - \th) \toD N_3(0,\Sigma), \quad n\to \infty
\] 
with $\G_n = {\rm diag} (1/\sqrt{nh_n}, 1/\sqrt{n}, 1/\sqrt{nh_n})$ and a diagonal matrix $\Sigma=\mathrm{diag}(\Sigma_1,\Sigma_2,\Sigma_3)$ (see, e.g., Shimizu and Yoshida \cite{sy06}). 
In this case, we have  $\g_{n*} = 1/\sqrt{nh_n}$, and Theorem \ref{thm:int-V} says that 
\[
\sqrt{nh_n}[H(\wh{\th}_n) - H(\th_0)] \toD N\l( 0, C_{\th_0}^\top \mathrm{diag}(\Sigma_1,0,\Sigma_3) C_{\th_0}\r)
\]
where 
\[
C_\th = \l(\int_0^T e^{-\d t}\E[Y^\th_t]\,\df t\r) =:(C_\th^1,C_\th^2,C_\th^3)^\top; 
\]
with $C_\th^2=0$ and 
\begin{align*}
C_\th^1 &= \frac{\eta - \mu x}{\mu(\d +\mu)^2}\l[1 - (\mu + \d)e^{-(\mu + \d)T - e^{-(\mu+\d)T}}\r] + \frac{\eta}{\d \mu^2}(1-e^{-\d T})  \\ &\qquad  + \frac{\eta}{\mu^2(\mu + \d)}(1-e^{-(\mu + \d)T}); \\
C_\th^3 &= \frac{1}{\mu}\l[\frac{1}{\d}(1-e^{-\d T}) - \frac{1}{\mu + \d}(1-e^{-(\mu + \d)T})\r]. 
\end{align*}
\end{example}

\section{Numerical experiments}
In this section, we shall illustrate the result of Theorem \ref{thm:final} by an example described in Introduction. That is, evaluating the statistical error of Monte Carlo estimation of a European call option based on a diffusion process of the underlying asset. 
In the experiments, we assume that the asset process is observed discretely in time with the {\it small noise asymptotics}, which philosophically corresponds to a kind of long term observations; see, e.g., Shimizu \cite{s17}, Section 2.3. 
We first compute estimators of unknown parameters in the process from discrete samples, and compute the price of the European call option by Monte Carlo simulations based on the estimated process, and evaluate the statistical error. 

In Section \ref{sec:sim-setting}, we will describe a general framework of the simulations as well as how to construct estimators of unknown parameters in the process. In Section \ref{sec:BS}, we will particularly consider the Black-Scholes model, and investigate the asymptotic distribution of the estimated European call price by the Monte Carlo method. 

We used the {\tt yuima} package in {\tt R} for simulating paths and discrete sampling from diffusion processes; see Brouste {\it et al.} \cite{yuima} for details.  

\subsection{European call-type functionals under the small noise diffusions}\label{sec:sim-setting}

We assume that the stochastic process $X^\th=(X^\th_t)_{t \ge 0}$ satisfies the following stochastic differential equation: for $\th=(\mu,\s)$ and $\e>0$,  
\begin{align}
\df X_t^{\th_0,\e} = b(X^{\th_0,\e}_t, \mu_0)\,\df t + \e\cdot a(X^{\th_0,\e}_t,\s_0) \,\df W_t,\quad 
X^{\th_0,\e}_0=x, 
\label{small-diff}
\end{align}
and suppose that we observe $X^{\th_0,\e}$ discretely at time points $t_k=k/n\, (k=0,1,\dots, n)$ in $[0,1]$-interval, and write the samples $X_{t_k}\,(k=0,1,\dots, n)$ and denote by $\D_kX = X_{t_k} - X_{t_{k-1}}$. 
Our purpose is to estimate the following expected functional with parameters $r, K, T>0$:  
\begin{align}
{\cal C}_{T,K}^\th :=\E\l[e^{-r T}\max \{X^\th_T - K, 0\}\r]. \label{E-call}
\end{align}
Under the {\it small noise asymptotics}: 
\[
(\e\sqrt{n})^{-1} = O(1), \quad \e \to 0, \quad n\to \infty, 
\]
we can see that a minimum contrast estimator 
\begin{align}
\wh{\th}_n &= \arg\min_\th M_n(\th), \label{M-estimator}
\end{align}
with the contrast function 
\begin{align*}
M_n(\th) &=\sum_{k=1}^n  \l[\frac{n}{\e^2}\frac{\l(\D_kX - \frac{1}{n} b(X_{t_{k-1}},\mu)\r)^2}{a^2(X_{t_{k-1}},\s)} + \log a^2(X_{t_{k-1}},\s)\r], 
\end{align*}
is asymptotically normal: 
\begin{align}
\l(\e^{-1} (\wh{\mu}_n - \mu_0),\sqrt{n}(\wh{\s}_n - \s_0)\r) \toD N_2\l(0, I^{-1}_{\th_0}\r),\quad n\to \infty, 
\label{asymptotics}
\end{align}
where 
\[
I_\th = \mathrm{diag}\l( \int_0^1 \l(\frac{\n_\mu b(X^{\th,0}_s,\mu)}{a(X^{\th,0}_s,\s)}\r)^2\df s, \ 
\frac{1}{2}\int_0^1 \l(\frac{\n_\s a^2(X^{\th,0}_s,\s)}{a^2(X^{\th,0}_s,\s)}\r)^2\,\df s\r); 
\]
see S{\o}rensen and Uchida \cite{su03}. There is an another type of estimators as in Uchida \cite{u08}. 
See also Long {\it et al.} \cite{letal13} and Shimizu \cite{s17} if $X$ is a jump process. 

We would like to numerically demonstrate the result of Theorem \ref{thm:final} with 
\[
\vp_\th(x) = e^{-rT}\max\{x - K,0\}, 
\]
and discretely observed diffusions as in \eqref{small-diff}, but $\vp_\th$ is not differentiable at $x=K$. 
In order to apply the theorem, we can use the approximation as in Example \ref{ex:asian} by  
\[
\vp^\d(x) = \frac{e^{-rT}}{2} \l(\sqrt{(x-K)^2 + \d^2} + x - K\r).  
\]
Setting $H_\d(\th) = \E[\vp^\d(X^\th_T)]$, we have by Theorem \ref{thm:final} that
\begin{align}
\e^{-1} ({\cal C}^{\wh{\th}_n}_{T,K} - {\cal C}^{\th_0}_{T,K}) \toD N(0,C_{\th_0}^\top I^{-1}_{\th_0} \,C_{\th_0}), \label{normal}
\end{align}
under the asymptotics \eqref{asymptotics} as well as $\d\to 0$, 
where 
\[
C_{\th} = \E\l[\frac{Y^{\th,0}_T}{2}\l\{\mathrm{sgn}(X^{\th,0}_T - K) + 1\r\} \r], 
\]
with $\mathrm{sgn}(z) = \I_{\{z>0\}} - \I_{\{z<0\}}$ and the derivative process $Y^{\th,\e}$.  
We will compute $C_{\th_0}$ by Monte Carlo simulations later. 

In the next section, we shall numerically illustrate \eqref{normal} by a more concrete model. 
\begin{remark}\label{rem:P*}
We claim that our setting above is a bit different from the financial practical problem since we do not care about the expectation in \eqref{E-call}, which should be the one with respect to a {\it risk neutral probability} although we should consider the parameter estimation under the physical probability measure $\P$ in financial problem. However, we shall try it just to numerically confirm our theoretical results. 
\end{remark}

\subsection{The Black-Scholes model with small noise}\label{sec:BS}
For simulations, we shall consider the standard Black-Scholes assumption: 
\begin{align}
X_t^{\th_0,\e} = x_0 + \int_0^t \mu_0 X_s^{\th_0,\e}\,\df s + \e \int_0^t \s_0  X_s^{\th_0,\e}\,\df W_s. \label{BSmod}
\end{align}
We assume that $X^{\th_0,\e}$ is observed at $t_k = k/n\ (k=0,1,2,\dots)$, and set the true parameters 
\begin{align}
\th_0=(\mu_0,\s_0) = (0.2,1.0),\quad x_0=1.0,\quad \e = 1/\sqrt{n} \label{truth}
\end{align}
for generating discrete samples from $X^{\th_0,\e}$. Then the minimum contrast estimator \eqref{M-estimator} is given in explicit form: 
\begin{align}
\wh{\th}_n = (\wh{\mu}_n,\wh{\s}^2_n) = \l(\sum_{k=1}^n \frac{\D_kX}{X_{t_{k-1}}},\ 
\e^{-2}\sum_{k=1}^n \frac{(\D_k X - \frac{\wh{\mu}_n}{n}X_{t_{k-1}})^2}{X^2_{t_{k-1}}}\r) \label{estimators}
\end{align}
and the Fisher information matrix becomes 
\[
I_{\th_0} = \mathrm{diag}\l(\s_0^{-4}, 2\s_0^{-2}\r) = \mathrm{diag}(1,2). 
\]

The derivative process $Y^\th_t = (Y^1_t,Y^2_t)$ is given by 
\begin{align*}
\df Y^{1,\e}_t &= (\mu Y^{1,\e}_t + X_t^{\th,\e})\,\df t + \e\cdot \s Y^{1,\e}_t\,\df W_t, \\
\df Y^{2,\e}_t &= \mu Y^{2,\e}_t\,\df t + \e (\s Y^{2,\e}_t + X_t^{\th,\e})\,\df W_t
\end{align*}
with $Y^{\th,\e}_0 = (0,0)$. A sample path of $(X^{\th_0,\e}, Y^{1,\e}, Y^{2,\e})$ is given in Figure \ref{fig:path}.

\begin{figure}
\begin{center}
\includegraphics[height=8cm, width=10cm]{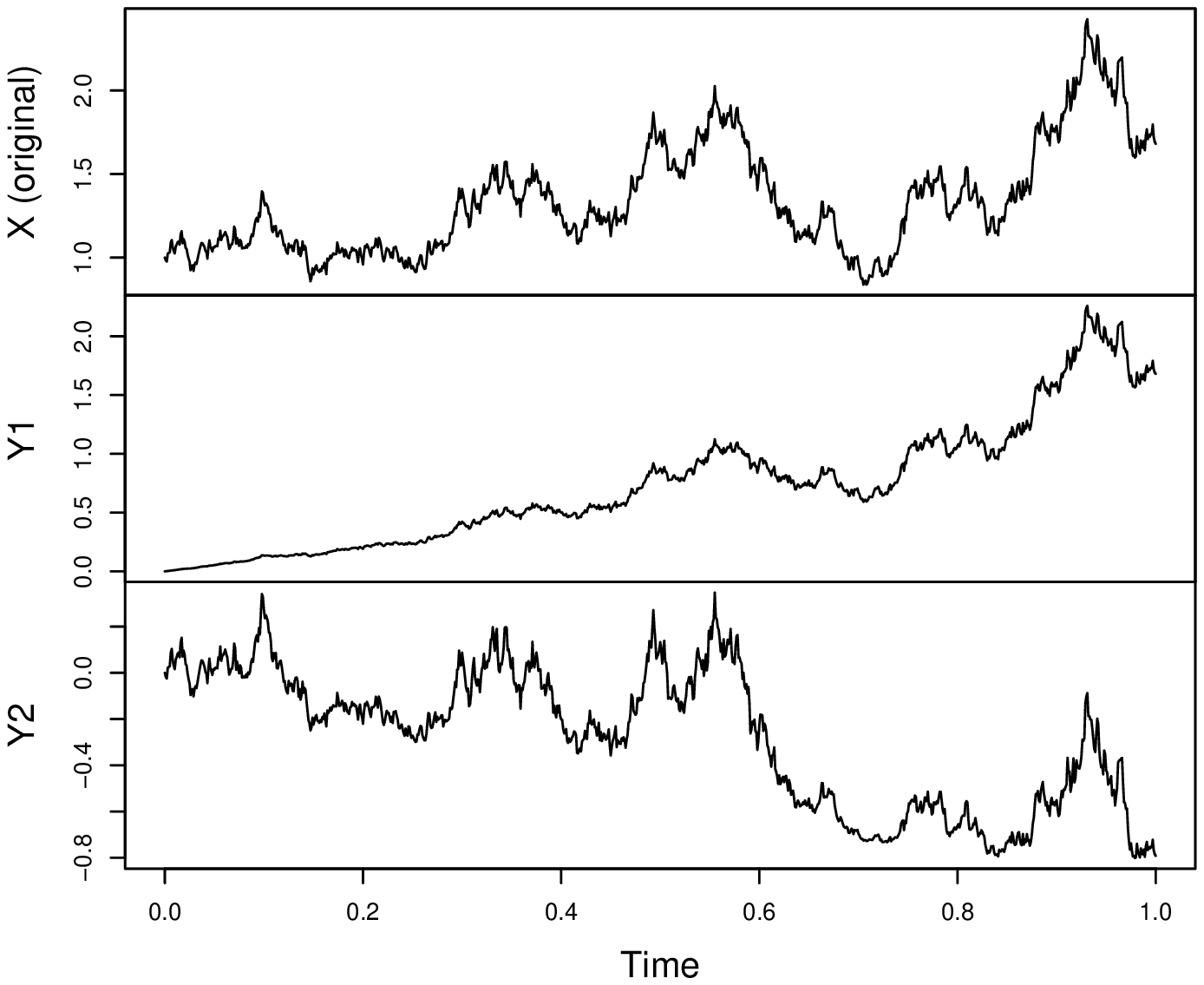}
\caption{Sample path of $(X^{\th_0,\e}, Y^{1,\e}, Y^{2,\e})$ with $(\mu_0, \s_0, x_0)=(0.5,1.0,1.0)$ with $\e =1$.}
\label{fig:path}
\end{center}
\end{figure}

We can find the explicit formula for ${\cal C}_{T,K}^{\th_0}$, \eqref{E-call} 
by using the well-known formula of the European call option price: 
\[
{\cal C}_{T,K}^{\th_0} = e^{-(r-\mu_0)T}\l[x \Phi(d_1) - K e^{-\mu_0 T}\Phi(d_2)\r], 
\]
where $\Phi(x) = \frac{1}{\sqrt{2\pi}}\int_{-\infty}^x e^{-z^2/2}\,\df z$ and 
\[
d_1 = \frac{\log (x/K) + ( r + \e^2\s_0^2/2)T}{\e\s\sqrt{T}};\quad 
d_2 = d_1- \e\s\sqrt{T}. 
\]
One may notice that the above formula is a bit different from the usual Black-Scholes formula. 
As is pointed out in Remark \ref{rem:P*},  we must be careful that the usual formula is under the {\it risk neutral probability}, under which the drift parameter $\mu_0$ corresponds to the interest rate. 
Since we are now ignoring the risk neutral transform, we have to make a minor modification to the formula. 

On the other hand, computing ${\cal C}_{T,K}^{\th_0}$ by Monte Carlo simulations, 
we see from Figure \ref{fig:C0} showing relative errors for Monte Carlo estimators, say ${\cal C}_{T,K}^*$, that 10,000 samples seem to be enough to compute ${\cal C}_{T,K}^{\th_0}$. 
Therefore, we also take 10,000 samples when we compute an estimator ${\cal C}_{T,K}^{\wh{\th}_n}$ by Monte Carlo simulations, below. 

\begin{figure}
\begin{center}
\includegraphics[height=8cm]{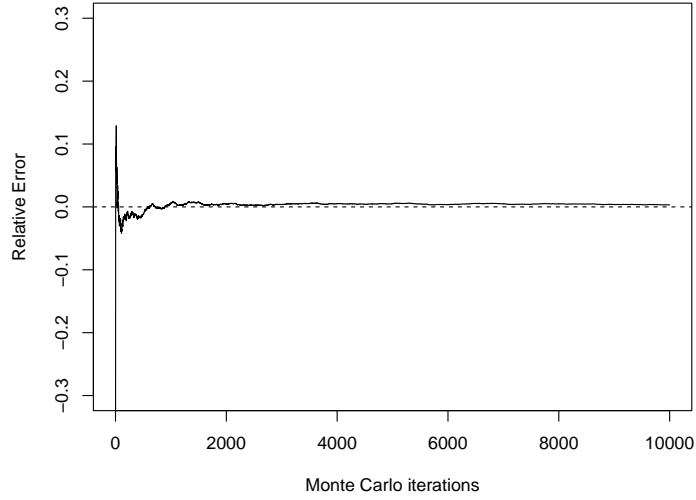}
\caption{The relative error $({\cal C}_{T,K}^* -{\cal C}_{T,K}^{\th_0})/{\cal C}_{T,K}^{\th_0}$, where ${\cal C}_{T,K}^*$ is a Monte Carlo estimator. The horizontal axis is the number of iteration for Monte Carlo simulations. }
\label{fig:C0}
\end{center}
\end{figure}

Numerical experiments are done by the following steps: 
\begin{itemize}
\item[(1)] Generate a path of \eqref{BSmod}, and get discrete samples. 
\item[(2)] Compute $\wh{\th}_n$ in \eqref{estimators}, and generate paths of $X^{\wh{\th}_n,\e}$. 
\item[(3)] Based on paths of (2), compute ${\cal C}_{T,K}^{\wh{\th}_n}$ by a Monte Carlo method: 
\[
{\cal C}_{T,K}^{\wh{\th}_n} \approx \frac{1}{B}\sum_{k=1}^B e^{-r T}\max\{X_T^{\wh{\th}_n,\e}(k) - K,0\}, 
\] 
where $X_T^{\wh{\th}_n,\e}(k)$ is the value of $X^{\wh{\th}_n,\e}_T$ starting from the initial value $X^{\wh{\th}_n,\e}_0=1.0$ for the $k$-th sample path. 
\item[(4)] Compare $N(0,1)$ and the (estimated) distribution of 
\begin{align}
\wh{Z}_n:=\e^{-1} ({\cal C}^{\wh{\th}_n}_{T,K} - {\cal C}^{\th_0}_{T,K})/\sqrt{C_{\th_0}^\top I^{-1}_{\th_0} \,C_{\th_0}}
\label{normed.estimator}
\end{align}
for different sample sizes: $n=50, 100,$ and $300$. 
\end{itemize} 
In our experiments, we put 
\[
T = 1.0,\quad K=0.75, \quad r = 0.05,\quad x=1.0. 
\]
Then,  we had 
\[
C_{\th_0}=\begin{pmatrix}1.64937\\ 0.00585\end{pmatrix}, 
\]
by Monte Carlo simulations, and that the asymptotic variance in \eqref{normed.estimator} is 
\[
C_{\th_0}^\top I^{-1}_{\th_0} \,C_{\th_0} = 1.649396. 
\]
We iterate the steps (1)--(4) 300 times, and show the histograms of $\wh{Z}_n$ and their estimated densities (by the kernel method, which were done by {\tt density()} in {\tt R}) as well as the normal QQ-plots in Figures \ref{fig:n=50} and \ref{fig:n=500}. Then, from their graphs, we can confirm that the asymptotic normality holds true even in the case where sample size $n$ is relatively small. 

\begin{center}
\begin{figure}[htbp]
\begin{tabular}{c}
\begin{minipage}{0.5\hsize}
\includegraphics[width=7cm, height=5.5cm]{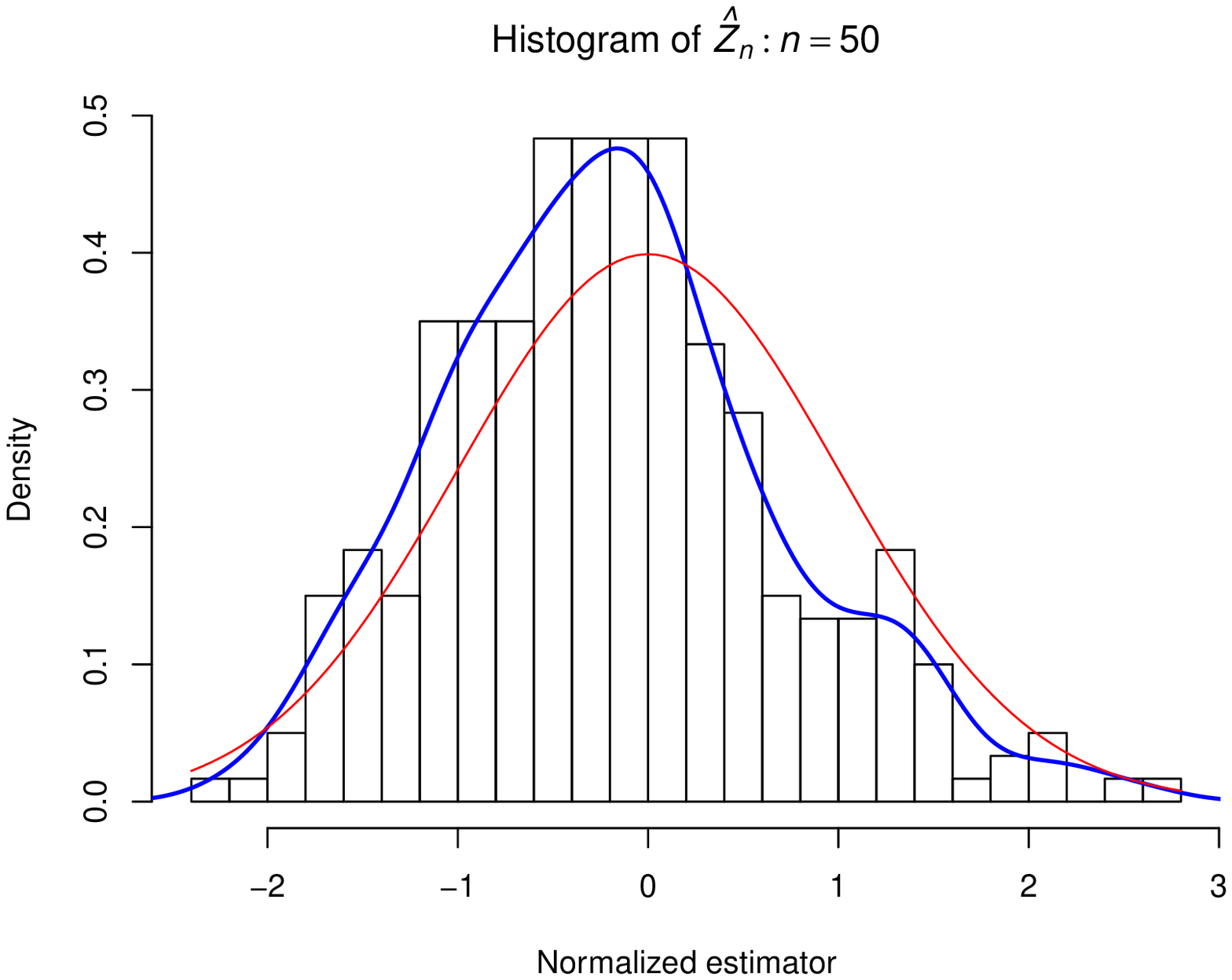}
\end{minipage}
\begin{minipage}{0.5\hsize}
\includegraphics[width=7cm, height=5.5cm]{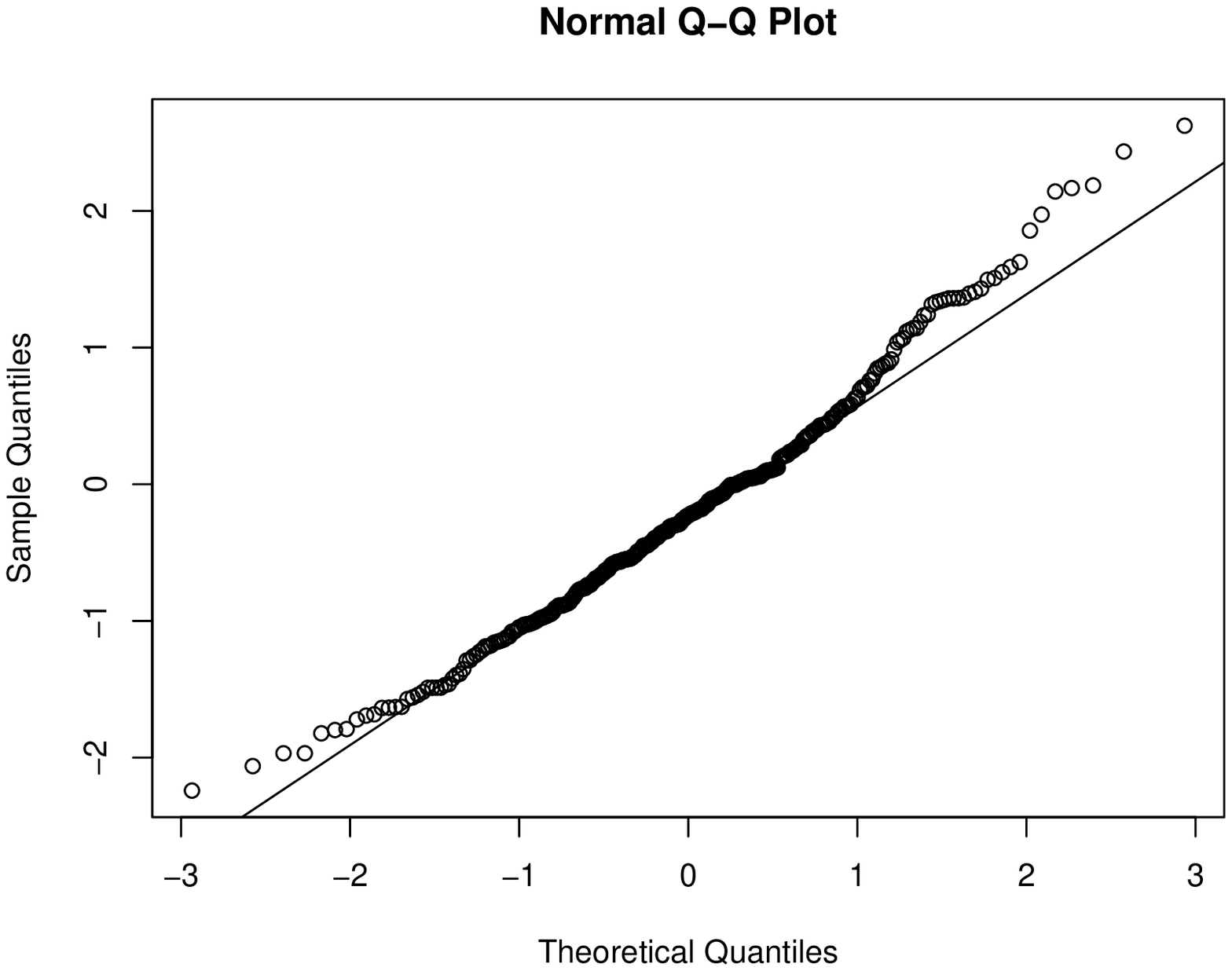}
\end{minipage}
\end{tabular}
\caption{As $n=50$: Histogram of $\wh{Z}_n$ and its estimated density (blue line) as well as the standard normal density (red line) (left); 
Normal QQ-plot (right) .}\label{fig:n=50}
\label{fig:theoretical}
\end{figure}
\end{center}

\begin{center}
\begin{figure}[htbp]
\begin{tabular}{c}
\begin{minipage}{0.5\hsize}
\includegraphics[width=7cm, height=5.5cm]{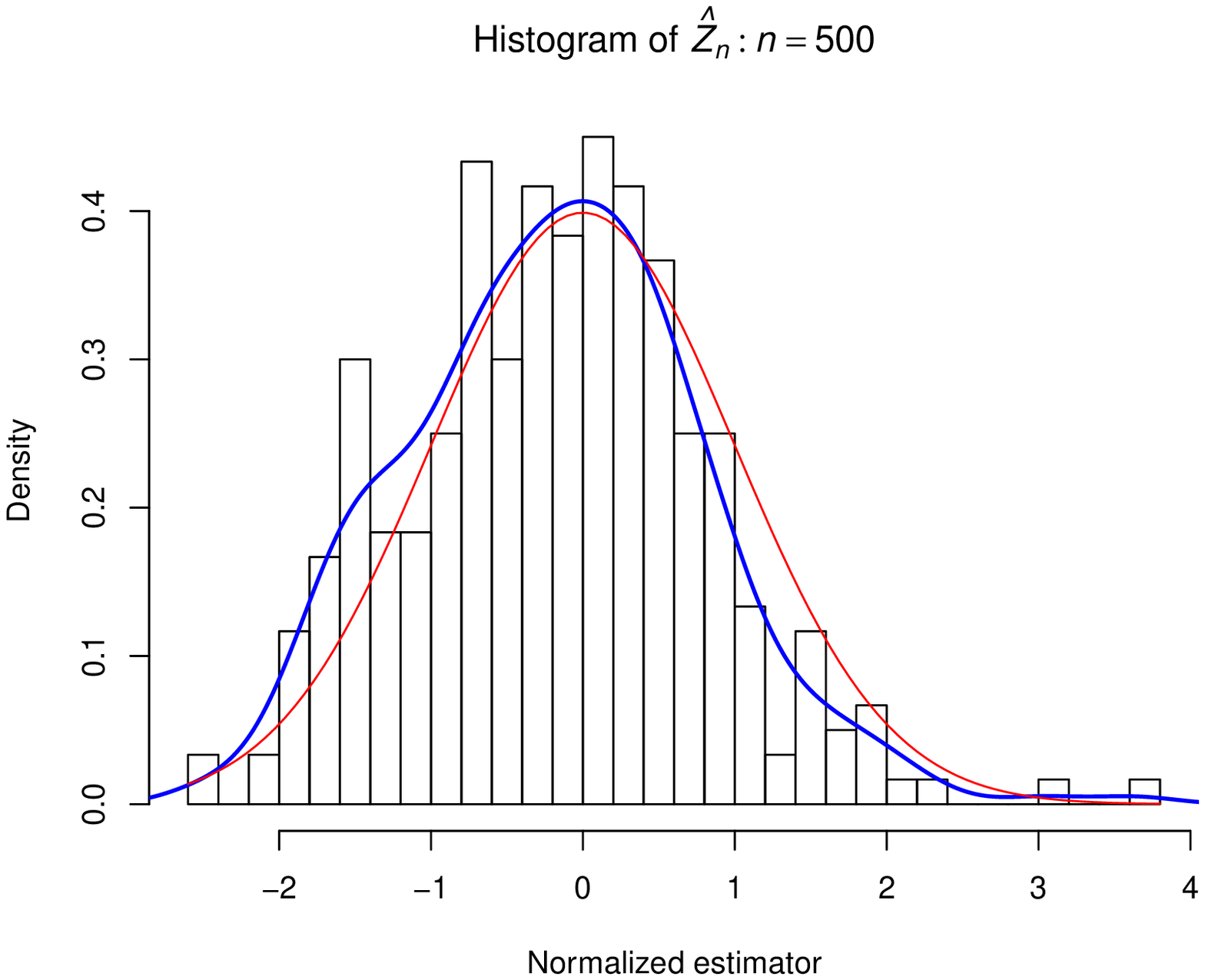}
\end{minipage}
\begin{minipage}{0.5\hsize}
\includegraphics[width=7cm, height=5.5cm]{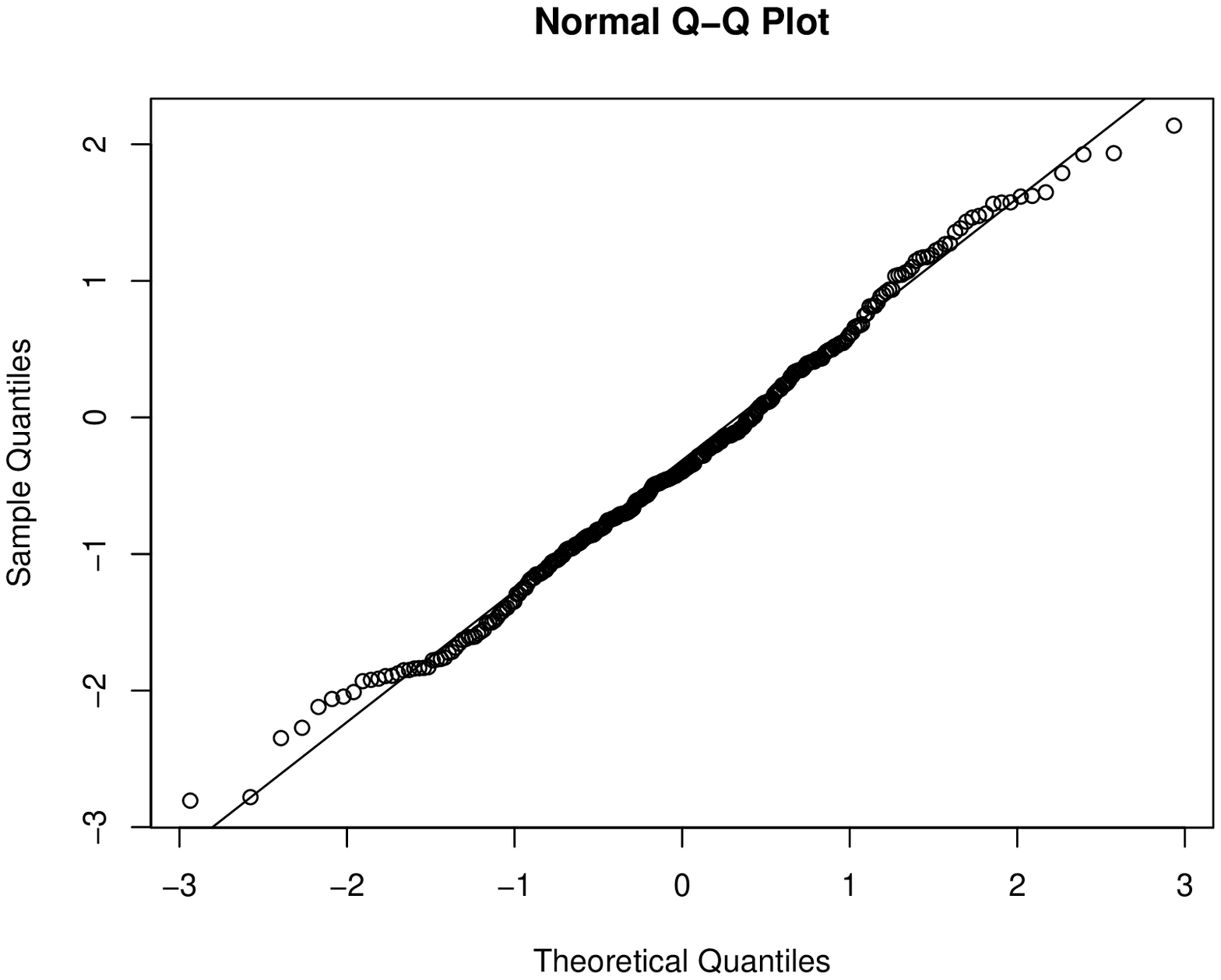}
\end{minipage}
\end{tabular}
\caption{As $n=500$: Histogram of $\wh{Z}_n$ and its estimated density (blue line) as well as the standard normal density (red line) (left); 
Normal QQ-plot (right).}\label{fig:n=500}
\label{fig:theoretical}
\end{figure}
\end{center}

\ \vspace{3mm}\\
\begin{flushleft}
{\bf\large Acknowledgements.}
The author expresses the sincere thanks to anonymous referees for detailed suggestions and proposals that makes the paper improve extensively.  
This research was partially supported by JSPS KAKENHI Grant-in-Aid for Scientific Research (A) \#17H01100; (B) \#18H00836 and JST CREST \#PMJCR14D7, Japan. 
\end{flushleft}



\begin{thebibliography}{99}

\bibitem{yuima} Brouste, A.; Fukasawa, M.;  Hino, H.;  Iacus, S.M.; Kamatani, K.; Koike, Y.; Masuda, H.;  Nomura, R.; Ogihara, T.; Shimizu, Y.; Uchida, M. and Yoshida, N. (2014). The YUIMA project: A computational framework for simulation and inference of stochastic differential equations, {\it Journal of Statistical Software}, {\bf 57}, (4),1--51. 

\bibitem{cg07} Chen, N. and  Glasserman, P. (2007). Malliavin greeks without Malliavin calculus, {\it Stochastic Processes and their Applications},  {\bf 117}, 1689--1723. 

\bibitem{dj06}  Davis, M.H.A. and Johansson, M. P. (2006). Malliavin Monte Carlo Greeks for jump diffusions,  {\it Stochastic Processes and their Applications}, {\bf 116}, 101--129. 

\bibitem{fs13} Feng, R. and Shimizu, Y. (2013). On a generalization from ruin to default in a L\'evy insurance risk model. {\it Methodol. Comput. Appl. Probab.}, {\bf 15}, (4), 773--802. 

\bibitem{fetal99} Fourni\'e, E.; Lasry, J.;  Lebuchoux, J;  Lions, P. and Touzi, N. (1999). Applications of Malliavin calculus to Monte Carlo methods in finance, {\it Finance and Stochastics}, {\bf 3}, 391--412. 

\bibitem{fetal01} Fourni\'e, E.; Lasry, J.;  Lebuchoux, J;  Lions, P. and Touzi, N. (1999). Applications of Malliavin calculus to Monte Carlo methods in finance II, {\it Finance and Stochastics}, {\bf 5}, 201--236. 

\bibitem{g09} Gentle J.E. (2009). Monte Carlo Methods for Statistical Inference. In: {\it Computational Statistics. Statistics and Computing}. Springer, New York, NY. 

\bibitem{gl04} Glasserman, P. (2004). {\it Monte Calro Methods in Financial Engineering}, Springer-Verlag, New York. 

\bibitem{gl10} Glasserman, P. and Liu, Z. (2010). Estimating greeks in simulating L\'evy-driven models, {\it The Journal of Computational Finance}, {\bf 14}, (2), 3--56. 

\bibitem{gk03} Gobet, E. and Kohatsu-Higa, A. (2003). Computation of Greeks for barrier and lookback options using Malliavin calculus, {\it Electronic Communications in Probability}, {\bf 8},  51--62.

\bibitem{ks91} Karatzas, I. and Shreve, S. E. (1991). {\it Brownian Motion and Stochastic Calculus}, 2nd. ed. Springer-Verlag, New York. 

\bibitem{km04} Kohatsu-Higa, A. and Montero, M. (2004). Malliavin calculus in finance, {\it Handbook of Computational and Numerical Methods in Finance}, Birkhauser, 111--174.

\bibitem{k04} Kutoyants, Yu.  (2004). {\it Statistical Inference for Ergodic Diffusion Processes}, Springer-Verlag London Limited.  

\bibitem{letal13} Long, H., Shimizu, Y. and Sun, W. (2013). Least squares estimators for discretely observed stochastic processes driven by small L\'evy noises. {\it J. Multivariate Analysis},  {\bf 116}, 422--439.

\bibitem{m92} Mammen, E. (1992). {\it When Does Bootstrap Work?: Asymptotic Results and Simulations}, Springer-Verlag, New York. 

\bibitem{p05} Protter, P. E. (2005). {\it Stochastic Integration and Differential Equations}, 2nd ed. Springer-Verlag,  Berlin, Heidelberg.  

\bibitem{rc04} Robert, C. and Casella, G. (2004). {\it Monte Carlo Statistical Methods}, Springer, New York, NY. 

\bibitem{sy06} Shimizu, Y. and Yoshida, N. (2006). Estimation of parameters for diffusion processes with jumps from discrete observations, {\it Statist. Infer. Stoch. Proc.}, {\bf 9}, (3), 227--277.

\bibitem{s17} Shimizu, Y. (2017). Threshold estimation for stochastic processes with small noise, {\it Scandinavian Journal of Statistics}, {\bf 44}, 951--988, 

\bibitem{su03} S{\o}rensen, M. and  Uchida, M. (2003). Small diffusion asymptotics for discretely sampled stochastic differential equations, {\it Bernoulli},  {\bf 9},1051--1069.

\bibitem{s83} Suri, R. (1983). Implementation of sensitivity calculations on a Monte Carlo experiment. {\it Journal of Optimization Theory and Applications}, {\bf 40}, 625--630.

\bibitem{u08} Uchida, M. (2008). Approximate martingale estimating functions for stochastic differential equations with small noises, {\it Stochastic Process. Appl.},  {\bf 118}, 1706--1721. 

\end{thebibliography}
\end{document}